\documentclass[12pt]{amsart}

\date{\today}
\usepackage[all]{xy}
\usepackage{fullpage} 
\usepackage[alphabetic,msc-links]{amsrefs}
\usepackage{amssymb, mathrsfs}
\usepackage{amsmath}
\usepackage{mathtools} 
\usepackage{xcolor}
\usepackage{makecell}
\setcellgapes{3.5pt}
\makegapedcells

\usepackage{colonequals}
\usepackage[overload]{textcase}
\usepackage{textgreek}
\usepackage{tikz-cd}
\usepackage[pdftex]{lscape}
\usepackage{mleftright}
\makeatletter
\newcommand{\vo}{\vec{o}\@ifnextchar{^}{\,}{}}
\makeatother

\theoremstyle{plain}
\newtheorem{thm}{Theorem}
\newtheorem{lem}[thm]{Lemma}
\newtheorem{prop}[thm]{Proposition}

\newtheorem{conj}[thm]{Conjecture}

\newtheorem{cor}[thm]{Corollary}
\theoremstyle{definition}
\newtheorem{defe}[thm]{Definition}
 
\theoremstyle{remark}
\newtheorem{rem}[thm]{Remark}

\newtheorem{ass}[thm]{Assumption}

\definecolor{red}{rgb}{1,0,0}
\definecolor{orange}{rgb}{1,0.5,0}
\definecolor{purple}{rgb}{.5,.2,.8}
\definecolor{blue}{rgb}{.2,.2,.8}
\definecolor{green}{rgb}{.4,.6,.4}

\newcommand{\ra}{\rightarrow}

\def\det{\mathrm{det}}

\def\bes{\begin{equation*}}  \def\ees{\end{equation*}} 
\def\bi{\begin{itemize}}   \def\ei{\end{itemize}}
\def\ba{\begin{eqnarray}} \def\ea{\end{eqnarray}}    
\def\bl{\begin{align}}    \def\el{\end{align}}       
\def\bls{\begin{align*}}    \def\els{\end{align*}}

\newcommand{\bC}{\mathbb{C}}
\newcommand{\bCt}{\mathbb{C}^\times}

\newcommand{\bX}{{X}}
\newcommand{\bA}{{A}}
\newcommand{\bY}{{Y}}
\newcommand{\GL}{\mathrm{GL}}

\newcommand{\bR}{R}
\newcommand{\bZ}{\mathbb{Z}}

\newcommand{\cT}{\mathcal{T}}

\newcommand{\PGL}{\mathrm{PGL}} 
\newcommand{\Sp}{\mathrm{Sp}}

\newcommand{\SO}{\mathrm{SO}}
\newcommand{\Irr}{\mathrm{Irr}}

\newcommand{\Hom}{\mathrm{Hom}}

\newcommand{\cO}{\mathcal{O}}
\newcommand{\cN}{\mathcal{N}}

\newcommand{\Fq}{\mathbb{F}_q}
\newcommand{\Fqr}{\mathbb{F}_{q^r}}

\newcommand{\fg}{\mathfrak{g}}
\newcommand{\ft}{\mathfrak{t}}

\newcommand{\fl}{\mathfrak{l}}

\newcommand{\chG}{\check{G}}
\newcommand{\chT}{\check{T}}
\newcommand{\bbR}{\mathbb{R}}


\begin{document} 
\title{Counting points on generic character varieties}

\author{Stefano Giannini$^1$}
\author{Masoud Kamgarpour$^2$}
\author{GyeongHyeon Nam$^3$}
\author{Bailey Whitbread$^4$}

\address{$^1$School of Computer and Mathematical Sciences, The University of Adelaide, Australia}
\email{\href{mailto:a1981079@adelaide.edu.au}{a1981079@adelaide.edu.au}}

\address{$^2$School of Mathematics and Physics, The University of Queensland, Australia}
\email{\href{mailto:masoud@uq.edu.au}{masoud@uq.edu.au}}

\address{$^3$Department of Mathematics and System Analysis, Aalto University, Finland}
\email{\href{mailto:gyeonghyeon.nam@aalto.fi}{gyeonghyeon.nam@aalto.fi}}

\address{$^4$School of Mathematics and Statistics, The University of Sydney, Australia}
\email{\href{mailto:b.whitbread@usyd.maths.edu.au}{b.whitbread@usyd.maths.edu.au}}

\begin{abstract} We count points on character varieties  associated with punctured surfaces and regular semisimple generic conjugacy classes in reductive groups. We find that the number of points are palindromic polynomials. This suggests a $P=W$ conjecture for these varieties. We also count points on the corresponding additive character varieties and find that the number of points are also polynomials, which we conjecture have non-negative coefficients. These polynomials can be considered as the reductive analogues of the Kac polynomials of comet-shaped quivers. 
\end{abstract}

\maketitle

\tableofcontents

\section{Introduction} 
\subsection{Overview} 
 
 Character varieties of surface groups play a central role in diverse areas of mathematics and physics, such as gauge theory \cite{AtiyahBott83, HitchinProceedingofLMS87}, mirror symmetry \cite{HauselThaddeusInventiones2003}, non-abelian Hodge theory \cites{Simpson92, Simpson94i}, and the geometric Langlands program \cites{BeilinsonDrinfeld, BenZviNadler}. Consequently, the topology and geometry of these varieties have been the focus of active research for decades. In their groundbreaking works \cite{HRV08, HLRV11}, Hausel, Letellier, and Rodriguez-Villegas counted points on $\GL_n$-character varieties,  leading to deep conjectures about the geometry of these varieties. 
 
 Many of these conjectures have since been proven. Notable achievements include the determination of the Betti numbers \cite{SchiffmannAnnals2016, MellitInventiones21, Mellit20}, the establishment of curious Poincaré duality and curious Hard Lefschetz \cite{MellitDecomposition}, and the recent resolution of the $P=W$ conjecture \cite{hausel2022p, maulik, HoskinPW}. 

In \cite{HLRV11}, the authors also considered the \emph{additive} version of character varieties, where the natural commutation relation is replaced by its Lie algebraic analogue. According to the Riemann--Hilbert correspondence, points on the usual (multiplicative) character varieties correspond to flat connections on vector bundles. When $g=0$, the points on the additive character variety correspond to flat connections on the \emph{trivial} bundle. Thus, when $g=0$, the additive character variety is a dense open subvariety of the usual character variety, sometimes known as the \emph{open de Rham space} \cite{HWW23}. 

This direct relationship between additive and multiplicative character varieties is lost in higher genus; nevertheless, it was conjectured in \cite{HLRV11} that the cohomology of the additive character variety captures the pure part of the cohomology of the multiplicative character variety in all genus. As evidence, the authors proved that additive character varieties are pure, i.e., their cohomology carries a pure Hodge structure. This result relied on the beautiful bridge built by Crawley-Boevey between additive character varieties and quiver varieties associated to comet-shaped quivers \cite{CBIndecomposable}. Under this correspondence, the number of points on the additive character variety matches the Kac polynomials of the corresponding quiver. Utilising these ideas, Hausel, Letellier, and Rodriguez-Villegas proved Kac's conjecture for arbitrary quivers \cite{HLVKacAnnals2013}.

Despite significant advances in understanding the geometry of character varieties associated with $\GL_n$ (and more generally, groups of type $A$), there is limited understanding of character varieties associated with arbitrary reductive groups $G$. Understanding these general cases is crucial for applications in fields such as Langlands duality and mirror symmetry, where the Langlands dual group plays a central role. 

In this paper, we count points on both multiplicative and additive character varieties associated with punctured surface groups and regular semisimple conjugacy classes in arbitrary reductive groups. Our work builds upon and generalises the contributions of Camb\`{o} \cite{Cambo17}, who studied $\mathrm{Sp}_{2n}$-character varieties; \cite{BK22}, who investigated character stacks associated with compact surface groups; and \cite{KNP25}, who examined character varieties in presence of regular unipotent conjugacy classes at the punctures.

One of our main results is that the number of points of multiplicative character varieties over finite fields are palindromic polynomials. This hints at a $P=W$ conjecture, a curious Poincaré duality, and a curious Hard Lefschetz property for these varieties.\footnote{To make further progress in this direction, one needs a reductive analogue of Markman's result \cite{MarkmanGenerators2001} that the tautological classes (and those coming from the underlying surface) generate the cohomology.} Another important contribution of this paper is an explicit polynomial formula for the number of points of additive character varieties.  The resulting polynomials can be considered as the reductive analogues of the Kac polynomials of comet-shaped quivers. Calculations suggest these polynomials have non-negative coefficients\footnote{See \cite{Bailey} for the computer code written using CHEVIE \cite{GHLMP96,Michel15}.}; however, lack of direct connection with quiver varieties makes proving this a non-trivial task.

Finally, one of the main contributions of \cite{HRV08, HLRV11} was to relate the number of points of $\GL_n$-character varieties with Macdonald polynomials and use this to provide a conjectural formula for the mixed Hodge polynomials of these varieties.\footnote{As far as we know, these conjectures (and the purity conjecture mentioned above) are still open.} Formulating analogous conjectures in the reductive setting is an interesting open problem.

\newpage

\subsection{Generic character varieties} \label{s:main}

Let $g$ and $n$ be integers satisfying $g\geq 0$ and $n\geq 1$. Let  $\Gamma=\Gamma_{g,n}$ be the fundamental group of an orientable surface with genus $g$ and $n$ punctures:
\[
\Gamma = {\langle a_1,b_1,\ldots,a_g,b_g,c_1,\ldots,c_n\,|\, {[a_1,b_1]\cdots[a_g,b_g]c_1\cdots c_n}=1\rangle}.
\]
Let $G$ be a connected split reductive group over $\Fq$ with split maximal torus $T$ and centre $Z$.  Let $C=(C_1,\ldots, C_n)$ be a tuple of conjugacy classes of $G$. 
 The \emph{representation variety} associated to $(\Gamma, G, C)$ is defined by 
\[
\bR :=  \bigg\{(a_1,b_1,\ldots ,a_g,b_g, c_1,\ldots ,c_n)\in G^{2g}\times \prod_{i=1}^n C_i\ \bigg|\  
\prod_{i=1}^g [a_i,b_i]\prod_{i=1}^n c_i = 1 \bigg\}.
\]
This is an affine scheme of finite type over $\Fq$ with an action of $G$ by conjugation. The associated character variety is the GIT quotient
\[
\bX:=\bR/\!\!/G. 
\]
We also have the corresponding character stack 
\[ [\bX]:=[\bR/(G/Z)].
\]

\begin{ass} \label{a:mainAss}
We will be working under the following assumptions: 
\begin{itemize} 
\item The centre $Z=Z(G)$ is connected and the ground characteristic is very good for $G$. 
\item $\prod_{i=1}^n C_i\subseteq [G,G]$, otherwise $\bR$ would be empty. 
\item Each conjugacy class $C_i$ contains an element $S_i\in T(\Fq)$ which is strongly regular; i.e., $C_G(S_i)=T$ (cf.\ \cite{Steinberg65,Humphreys95}).  
\item The tuple $C=(C_1,\ldots ,C_n)$ is \emph{generic} in the sense of \S \ref{s:generic}.
\item $2g+n\geq 3$, otherwise the generic assumption implies that $X$ is empty. 
\end{itemize} 
\end{ass}

As explained in \S \ref{s:genericVar}, 
the generic assumption ensures that the action of $G/Z$ on $\bR$ has finite \'etale stabilisers and that $\bR$ is smooth and equidimensional. Thus, $[X]$ is a smooth Deligne--Mumford stack. Moreover, $X$ and $[X]$ have the same number of points over a finite field, and are equidimensional of pure dimension
\[
\dim(\bX)=(2g-2+n)\dim(G)+2\dim(Z)-n\cdot \dim(T).
\]

Before stating our main theorem, we need to recall some facts about polynomial count varieties and $E$-polynomials. We refer the reader to  \cite{HRV08} or \cite{LetellierVillegas} for further details. A variety $S$ over $\Fq$ is said to be \emph{polynomial count} if there exists a polynomial $f_S(t)\in \mathbb{C}[t]$ such that 
\[
|S(\mathbb{F}_{q^m})| = f_S(q^m),\qquad \forall m\in \mathbb{Z}_{>0}. 
\]
In this case, $f_S$ is known as the \emph{counting polynomial} of $S$. An important fact is that if $S$ is polynomial count then its counting polynomial equals the \emph{$E$-polynomial} of $S$. The latter is a certain specialisation of the mixed Poincar\'e polynomial of $S$ (defined using \'etale cohomology). Since the $E$-polynomial has integer coefficients, it follows that the counting polynomial must belong to $\mathbb{Z}[t]$. 

The $E$-polynomial gives useful topological information about the variety. For instance, the leading coefficients of the $E$-polynomial equals the number of irreducible components of top dimension, and that the value of the $E$-polynomial at $1$ equals the Euler characteristic. 

Our main result states that generic character varieties are, after a finite base change, polynomial count:

\begin{thm} \label{t:countX}  There exists a polynomial $f_X \in \mathbb{Z}[t]$ (given explicitly in \S \ref{s:countXPrecise}) and a positive integer $r$ such that 
\[
|X(\mathbb{F}_{q^{rm}})| = f_X(q^{rm}),\qquad \forall m\in \mathbb{Z}_{>0}. 
\] 
\end{thm} 
The starting point is the observation that since $X$ is the coarse moduli space of the Deligne-Mumford stack $[X]$, we have $|X(\Fq)|=|[X](\Fq)|$. An application of Lang's theorem then implies 
\[
\displaystyle |[X](\Fq)|=\frac{|R(\Fq)||Z(\Fq)|}{|G(\Fq)|}.
\] 
Next, the Frobenius Mass Formula allows one to write $|\bR(\Fq)|$ in terms of  irreducible complex representations of the finite group $G(\Fq)$. We then use results of Deligne and Lusztig \cite{DL76} to analyse this formula and prove the theorem; see \S \ref{s:countX} for details. 

Let $\overline{X}$ denote the base change of $X$ to $\overline{\mathbb{F}_q}$ and $E(\overline{X})$ the $E$-polynomial of $\overline{X}$. The previous theorem implies that $E(\overline{X})=f_X$. By analysing the counting polynomial $f_X$, we will obtain:

\begin{cor} \label{c:X}
\begin{enumerate} 
\item[(a)] $E(\overline{X})$ is palindromic.  
\item[(b)] $|\pi_0(\overline{X})|=|\pi_0(Z(\check{G}))|$, where $\check{G}$ is the Langlands dual group. 
\item[(c)] If $g>1$, or $g=1$ and $Z$ is non-trivial, then $\chi(\overline{X})=0$. 
\end{enumerate}  
\end{cor}  

Expressions for the Euler characteristic in the case $g=0$ or $g=1$ and $Z=\{1\}$ is given in \S \ref{s:EulerX}. 
The first part of the above corollary provides evidence for the following:

\begin{conj} The analogue of  $\bX$ over the complex numbers is Hodge--Tate, and satisfies the curious Poincar\'{e} duality and curious  hard Lefschetz. Moreover, there is a version of the $P=W$ conjecture for $\bX$. 
\end{conj}

 When $G=\GL_n$, Mellit constructed a cell decomposition for $\bX$, thus providing a conceptual explanation of the fact that $\bX$ is polynomial count and proving in addition that it is Hodge--Tate and satisfies curious Hard Lefschetz and curious Poincar\'{e} duality \cite{MellitDecomposition}. We expect that $\bX$ has an analogous cell decomposition for arbitrary reductive group $G$.

\begin{rem} 
\begin{enumerate} \label{r:main}

\item[(i)] The polynomial $f_X$ depends only on $(G,g,n)$; i.e., it is independent of the conjugacy classes $C_i$ (provided they satisfy Assumption \ref{a:mainAss}).

\item[(ii)] We expect that the theorem holds without assuming that the centre $Z(G)$ is connected; however, the expression for $|X(\Fq)|$ becomes more complicated due to disconnected centralisers in the dual group. 
\item [(iii)] Our analysis shows that the stacky count  $|[\bX](\Fq)|$ is a polynomial in $q$ even if we drop the generic assumption. We expect that the regularity assumption is also not required for the polynomial property. We emphasise, however, that without the generic assumption, the variety $X$ is quite different from the stack $[X]$.

\item[(iv)] There are two reasons for the finite base change in the main theorem.  First, the inclusion $[G(\Fq), G(\Fq)] \subseteq [G,G](\Fq)$ may be proper (e.g. for $G=\PGL_n$). However, given $S\in [G,G](\Fq)$, there exists a positive integer $m$, such that  $S\in [G(\mathbb{F}_{q^m}), G(\mathbb{F}_{q^m})]$. 
Second, if $W'$ is a reflection subgroup of $W$, then $T^{W'}$ may be disconnected. The group of components $\pi_0=\pi_0(T^{W'})$ is a finite \'etale group which may be non-split (i.e. the action of Galois group may be non-trivial). This implies that $|\pi_0(\Fq)|$ depends on the residue of $q$ modulo $|\pi_0(\overline{\Fq})|$. However, as we shall see, the fact that $\mathrm{char}(k)$ is good for $G$ implies that $q$ is co-prime to $|\pi_0(\overline{\Fq})|$. Thus, there exists a positive integer $n$ such that 
\[
q^n \equiv 1 \mod |\pi_0(\overline{\Fq})|.
\]
It follows that $|\pi_0(\mathbb{F}_{q^n})|=|\pi_0(\overline{\Fq})|$.

\end{enumerate} 
\end{rem}

\subsection{The additive analogues}\label{s:additive}  We now consider the additive (or Lie-algebraic) analogues of the above discussions. Let $\fg$ be the Lie algebra of $G$ and  ${\cO=(\cO_1,\ldots, \cO_n)}$ a tuple of adjoint $G$-orbits on $\fg$. Define the \emph{additive representation variety}  by  
\[
\bA :=  \bigg\{(a_1,b_1,\ldots,a_g,b_g, c_1,\ldots ,c_n)\in \fg^{2g}\times \prod_{i=1}^n \cO_i\ \bigg|\  
\sum_{i=1}^g [a_i,b_i]+\sum_{i=1}^n c_i = 0 \bigg\}.
\]
The group $G$ acts on $\bA$ by the adjoint action. 
We call the resulting GIT quotient  
\[
\bY:=\bA/\!\!/G
\]
the \emph{additive character variety} associated to $(\Gamma, G, \cO)$.

We will be working under the obvious additive analogues of Assumption \ref{a:mainAss}. Then the action of $G/Z$ on $\bA$ has finite \'etale stabilisers, $\bA$ is smooth and equidimensional,  and $\bY$ is also equidimensional of the same dimension as $\bX$. 

Our main result states that the variety $Y$ is, after a quadratic base change, polynomial count:

\begin{thm} \label{t:countY}  There exists a polynomial $f_Y \in \mathbb{Z}[t]$ (given explicitly in \S \ref{s:countYPrecise}) such that 
\[
|Y(\mathbb{F}_{q^{2m}})| = f_Y(q^{2m}),\qquad \forall m\in \mathbb{Z}_{>0}. 
\] 
\end{thm}

The starting point of the proof is an additive analogue of the Frobenius formula which gives an expression for $|\bY(\Fq)|$ in terms of the Fourier transforms of $G$-invariant orbits on $\fg$. We then use a result of Kazhdan and Letellier to evaluate these Fourier transforms and compute $|Y(\Fq)|$; see \S \ref{s:countY} for details.

\begin{rem} The reason for the quadratic base change is that computing $|Y(\Fq)|$ involves the size of a Springer fibre over $\Fq$. It was shown by Springer that this is a polynomial if $q$ is large enough.  This was refined in \cite[Theorem 3.10]{GoodwinRohrle} where it was shown that  (under the assumption that $Z$ is connected and $p$ is good for $G$) the size of the Springer fibre is a polynomial in $q$, unless $G$ has a factor of type $E_8$, in which case a quadratic base change may be necessary. 
\end{rem}

Let $\overline{Y}$ be the base change of $Y$ to $\overline{\Fq}$. By analysing the counting polynomial $f_Y$, we find: 

\begin{cor} \label{c:Y}
\begin{enumerate} 
\item[(a)] The variety $\overline{Y}$ is connected.  
\item[(b)] If $G$ is simple of rank less than $7$, $0\leq g\leq 10$ and $1\leq n\leq 1000$, then the coefficients of the counting polynomial $f_Y$ are non-negative.
\end{enumerate} 
\end{cor}

Part (b) of this corollary is established with code written using CHEVIE \cite{GHLMP96,Michel15}; see \cite{Bailey} for the relevant code.
The counting polynomial can also be used to give expressions for the Euler characteristic; see \S \ref{s:EulerY}.

When $G=\GL_n$, the additive character variety $Y$ is isomorphic to a quiver variety and $|Y(\Fq)|$ is a polynomial with non-negative coefficients \cite{HLRV11}.  Combined with the previous corollary, this provides evidence for the following:

\begin{conj} \label{c:purity}
The cohomology $H_c^*(\overline{\bY},\mathbb{Q}_\ell)$ is pure; thus, $|Y(\Fq)|$ has non-negative coefficients. 
\end{conj}

If $G=\GL_n$ then $Y$ is isomorphic to a quiver variety and the coefficients of $|Y(\Fq)|$ were recently shown to be the graded dimensions of the associated graded Borcherd algebra \cite{Lucien}, proving a conjecture of Schiffmann \cite{Schiffmann18}. It would be interesting to find a similar interpretation for the coefficients of $|\bY(\Fq)|$ for general reductive groups.

 \subsection{Structure of the text} 
 In \S \ref{s:prelim}, we gather some facts about reductive groups and  root systems. In \S \ref{s:geometry}, we discuss basic results about the geometry of generic character varieties. The main computation regarding points on character varieties appears in \S \ref{s:countX} and its topological implications are discussed in \S \ref{s:topology}. The additive analogues of these results appears in \S \ref{s:countY}. Finally, in \S \ref{s:Examples}, we give explicit examples of counting polynomials for groups of low rank.

\subsection{Acknowledgement} We would like to thank Ivan Cheltsov, Anand Deopurkar, Jack Hall, Lucien Hennecart, Ian Le, Emmanuel Letellier, Paul Levy, Sebastian Schlegel Mejia, Jean Michel, Travis Schedler, and Behrouz Taji for helpful conversations. 

MK was supported by an Australian Research Council (ARC) discovery grant. GH was supported by the National Research Foundation of Korea (NRF). BW and SG were supported by Australian postgraduate scholarships. We would like to thank the Sydney Mathematical Research Institute for facilitating a visit by Emmanuel Letellier,  giving further impetus to this project. 

The material of \S \ref{s:countX} (resp.\ \S \ref{s:countY}) forms a part of the Master's thesis of BW (resp.\ SG).

\section{Recollections on root systems and reductive groups} \label{s:prelim}
In this section, we recall basic facts and notations about root systems and reductive groups. The material here is standard, though perhaps the facts about pseudo-Levi subgroups and quasi-isolated semisimple elements are not as widely known. For these notions, we refer the reader to \cite{MS03}, \cite{Bonnafe05}, and \cite{Taylor22}.

\subsection{Root systems} 
Let $\Phi$ be a (reduced crystallographic) root system in a finite dimensional Euclidean vector space $(V, (.\ ,.))$.  We assume that $\Phi$ spans $V$ and fix a base $\Delta$ of $\Phi$. Then $\Delta$ is a basis of the vector space $V$. Let $W$ be the Weyl group generated by the simple reflections $s_\alpha$, $\alpha \in \Delta$. 

\subsubsection{Subsystems} A root subsystem of $\Phi$ is a subset $\Psi\subseteq \Phi$ which is itself a root system. This is equivalent to the requirement that for all $\alpha, \beta\in \Psi$,  $s_\alpha(\beta)\in \Psi$. A subsystem $\Psi\subseteq \Phi$ is \emph{closed} if 
\[
\alpha, \beta\in \Psi \quad \textrm{and}\quad  \alpha+\beta\in \Phi\quad \implies \quad \alpha+\beta\in \Psi. 
\]
 Given a subset $S\subseteq \Phi$,  let $\langle S\rangle_{\bZ}$ denote the subgroup of $V$ generated by $S$.  The subsystem of $\Phi$ generated by $S$ is defined by 
  \[
  \langle S \rangle:=\langle S\rangle_{\bZ} \cap \Phi. 
  \]
Note that $ \langle S \rangle$ is a closed subsystem of $\Phi$.

\subsubsection{Dual root systems}\label{s:dual} For each root $\alpha\in \Phi$, define the coroot 
\[
\check{\alpha}:=\frac{2}{(\alpha, \alpha)}\alpha\in V.
\]
 The set of coroots forms the dual root system $\check{\Phi}\subset V$. If $\Psi$ is a subsystem of $\Phi$, then $\check{\Psi}$ is a subsystem of $\check{\Phi}$. However, if $\Psi$ is closed in $\Phi$, then $\check{\Psi}$ need not be closed in $\check{\Phi}$. For instance, $C_n\times C_n$ is a closed subsystem of $C_{2n}$ but if $n>1$, then $B_n\times B_n$ is not closed in $B_{2n}$. Indeed, $B_n\times B_n$ does not arise in the Borel--de Siebenthal algorithm applied to $B_{2n}$.  

\subsubsection{Levi subsystems}
 A \emph{Levi subsystem} (aka parabolic subsystem)
   of $\Phi$ is a subsystem of the form $\Phi \cap E$ where $E\subseteq V$ is a subspace. 
  A Levi subsystem of $\Phi$ is closed in $\Phi$. Moreover, every Levi subsystem of $\Phi$ is of the form $w\cdot \langle S\rangle$, where $S$ is a subset of $\Delta$ and $w\in W$.  Note that $\Psi$ is a Levi subsystem of $\Phi$ if and only if $\check{\Psi}$ is a Levi subsystem of $\check{\Phi}$.

\subsubsection{Isolated subsystems}\label{s:Isolated}
  A subset $S\subseteq \Phi$ is called \emph{isolated in $\Phi$} if $S$ is not contained in a proper Levi subsystem of $\Phi$. The following are equivalent: 
\begin{enumerate} 
\item $S$ is isolated in $\Phi$. 
\item $\check{S}$ is isolated in $\check{\Phi}$. 
\item $S$ does not lie in a proper subspace of $V$. 
\item $\langle S\rangle_{\mathbb{Z}}\otimes \bbR =V$.
\item The root systems $\langle S\rangle$ and $\Phi$ have the same rank. 
\item $\displaystyle \bigcap_{\alpha \in S} \mathrm{ker}(\alpha)=0$. 
\end{enumerate}

\subsubsection{Pseudo-Levi subsystems} Assume $\Phi$ is irreducible. Let $\theta$ be the highest root of $\Phi$, and 
\[
\widetilde \Delta:=\Delta \sqcup \{-\theta\}.
\]
 A pseudo-Levi subsystem is a subsystem of $\Phi$ of the form $w\cdot \langle S\rangle $, where $w\in W$ and $S$ is a subset of $\widetilde \Delta$.  More generally, if $\Phi=\Phi_1\sqcup \cdots \sqcup \Phi_r$ with $\Phi_i$ irreducible, then a subsystem $\Psi\subseteq \Phi$ is said to be a pseudo-Levi if $\Psi_i:=\Psi\cap \Phi_i$ is a pseudo-Levi inside $\Phi_i$ for all $i$. It is clear that every Levi is a pseudo-Levi. The converse is true in type $A$ but false otherwise. For instance, $C_2$ has a pseudo-Levi $A_1\times A_1$ which is not a Levi. 

\subsubsection{Isolated pseudo-Levis} 
 Let $\Psi$ be a pseudo-Levi subsystem of an irreducible root system $\Phi$.  Then the following are equivalent: 
\begin{enumerate} 
\item $\Psi$ is isolated in $\Phi$. 
\item $\Psi=w\cdot \langle S\rangle $ where $S\subset \tilde{\Delta}$ has the same size as $\Delta$.  
\end{enumerate} 
In other words, up to $W$-conjugation, isolated pseudo-Levis are those which are obtained by removing just a single element of $\widetilde{\Delta}$ (equivalently, a single node from the extended Dynkin diagram). We leave it to the reader to generalise these statements to non-irreducible root systems. Figure \ref{table:iso} lists the isolated pseudo-Levi subsystems of irreducible root systems.

\begin{figure}[h]\centering
\begin{tabular}{|c|c|} 
\hline Type of $\Phi$ & Isolated pseudo-Levi subsystems of $\Phi$ \\
\hline 
$A_n$ $(n\geq 1)$ & $A_n$ only \\ \hline 
$B_n$ $(n\geq 3)$ & \makecell{$B_n$, \ $D_n$, \ $D_{n-1}\times A_1'$, \ $D_r\times B_{n-r}$ ($2\leq r\leq n-2$)} \\\hline 
$C_n$ $(n\geq 2)$ & \makecell{$C_n$, \ $C_r \times C_{n-r}$ ($1\leq r \leq \frac{n}{2}$)} \\\hline 
$D_n$ $(n\geq 4)$ & \makecell{$D_n$, \ $D_r \times D_{n-r}$ ($2\leq r \leq \frac{n}{2}$)} \\\hline 
$G_2$ & \makecell{$G_2$, \ $A_2$, \ $A_1\times A_1'$} \\\hline 
$F_4$ & \makecell{$F_4$, \ $B_4$, \ $C_3\times A_1$, \ $A_3\times A_1'$, \ $A_2\times A_2'$} \\\hline 
$E_6$ & \makecell{$E_6$, \ $A_5\times A_1$, \ $A_2\times A_2\times A_2$} \\\hline 
$E_7$ & \makecell{$E_7$, \ $A_7$, \ $D_6\times A_1$, \ $A_5\times A_2$, \ $A_3\times A_3\times A_1$} \\\hline 
$E_8$ & \makecell{$E_8$, \ $D_8$, \ $A_8$, \ $E_7\times A_1$, \ $A_7\times A_1$, \\ 
$E_6\times A_2$, \ $D_5\times A_3$, $A_5\times A_2\times A_1$, \ $A_4\times A_4$} \\\hline 
\end{tabular}
\caption{Isolated pseudo-Levi subsystems. \label{table:iso}}
\end{figure}

\subsubsection{Endoscopy subsystems} 
  Let $\check{\Psi}$ be a pseudo-Levi subsystem of $\check{\Phi}$. Then the dual root system $\Psi$ is called an \emph{endoscopy subsystem} of $\Phi$. For instance, $B_n\times B_n$ is an endoscopy subsystem of $B_{2n}$. Note that this subsystem is not closed.

\subsubsection{Very good primes} We say the prime $p$ is good for $\Phi$ if $p$ does not divide the coefficients of 
the highest root of $\Phi$. Otherwise $p$ is bad for $G$. Bad characteristics are $p = 2$ if the root system is of type $B_n$, $C_n$ or $D_n$, $p=2,3$  in type $G_2$, $F_4$, $E_6$, $E_7$ and $p=2,3,5$ in type $E_8$.   We say $p$ is very good for $\Phi$ if $p$ is good for $\Phi$ and $p$ does not divide $n+1$ for every  irreducible component of type $A_n$ of $\Phi$. Note that $p$ is (very) good for $\Phi$ if and only if it is (very) good for $\check{\Phi}$.

\subsection{Reductive groups} \label{s:Reductive} Let $G$ be a connected reductive group over an algebraically closed field $k$. Let $T$ be a maximal torus of $G$, $B$ a Borel containing $T$, $U$ the unipotent radical of $B$, $W$ the Weyl group, $(X,  \Phi, \check{X}, \check{\Phi})$ the root datum,  $\langle \Phi \rangle_\bZ$  root lattice, and $\langle \check{\Phi} \rangle_\bZ$  the coroot lattice. Let $\Delta\subseteq \Phi$ (resp. $\check{\Delta}\subseteq \check{\Phi}$) denote the base determined by $B$ and $V:=X\otimes \bbR$.

\subsubsection{Dual group} 
Let $\chT:=\mathrm{Spec}\, k[X]$ be the dual torus and $\chG$ the Langlands dual group of $G$ over $k$. In other words, $\chG$ is the connected reductive group over $k$ with maximal torus $\chT$ and root datum $(\check{X}, \check{\Phi}, X, \Phi)$.

\subsubsection{Levi subgroups} A Levi subgroup of $G$ is defined to be the centraliser $L=C_G(S)$ of a torus $S\subseteq G$. If $S\subseteq T$ (in other words, if $L$ contains $T$), then $L$ is said to be a \emph{standard} Levi subgroup. In this case, the root system of $L$ is a Levi subsystem of $\Phi$. Every Levi subsystem arises in this manner. 

If $x\in \fg$ is a semisimple element, then the connected centraliser $C_{G}(x)^\circ$ is a Levi subgroup of $G$. Moreover, every Levi subgroup of $G$ arises in this manner. Note that if the characteristic of $k$ is very good for $G$, then $C_G(x)$ is connected \cite[Theorem 3.14]{Steinberg75}. 

\subsubsection{Parabolic subgroups}   \label{s:parabolic} A connected subgroup $P\subseteq G$ is called parabolic if the quotient $G/P$ is proper. Every parabolic subgroup has a Levi decomposition $P=L\ltimes N$, where $N$ is the maximal closed connected normal unipotent subgroup of $P$ and $L$ is a Levi subgroup of $G$. Every Levi subgroup of $G$ arises in this manner.  

\subsubsection{} In what follows, we will be using the following facts: 
\begin{enumerate} 
\item 
If $s\in P$ is semisimple, then $s$ is $P$-conjugate to an element of $L$. Indeed, $s$ lies in some maximal torus $T'$ of $P$ and since all maximal tori of $P$ are $P$-conjugate, we can choose a $P$-conjugate of $T'$ which lies in $L$. 

\item If $l\in L$, then $l\in [L,L]$ if and only if $l\in [P,P]$. This follows from the fact that $[P,P]=[L,L]\ltimes N$. 
\end{enumerate}

\subsubsection{Pseudo-Levi subgroups}  Let $x\in G$ be a semisimple element. The connected centraliser $G_x^\circ=C_G(x)^\circ$ is called a pseudo-Levi subgroup of $G$.  The dual group is called an endoscopy group for $G$.  If $x\in T$, then $G_x^\circ$ is called a \emph{standard} pseudo-Levi subgroup of $G$ and its dual is called a \emph{standard} endoscopy group for $\check{G}$. 

\subsubsection{} If $x\in T$, then the root system of $G_x^\circ$ is given by 
\[
{\Phi}_x:=\{{\alpha} \in {\Phi} \, |\, {\alpha}(x)=1\}.
\]
This is a pseudo-Levi subsystem of $\Phi$. Moreover, every pseudo-Levi subsystem arises in this manner \cite{Deriziotis}. By a theorem of Steinberg, $\pi_0(G_x)$ is isomorphic to a subgroup of $\pi_1([G,G])$; thus, if $[G,G]$ is simply connected, then $G_x$ is connected \cite[\S 2]{Steinberg75}.

\subsubsection{Quasi-isolated elements} \label{s:isolated} A semisimple element $x\in G$ is called \emph{quasi-isolated} if $G_x$ is not in a proper Levi subgroup of $G$.  It is called isolated if $G_x^\circ$ is not in any proper Levi subgroup of $G$. Thus, $x\in T$ is isolated if and only if $\Phi_x$ is an isolated (pseudo-Levi) subsystem of $\Phi$. If $x$ is quasi-isolated, then $G_x$ is an irreducible subgroup of $G$; i.e., it does not belong to a proper parabolic. By Schur's lemma (cf.\ \cite[Proposition 15]{Sikora}) $Z(G_x)/Z(G)$ is finite.

\subsubsection{Pretty good primes}\label{sss:prettygood} Following \cite{Herpel}, we say a prime $p$ is pretty good for $G$ if for every subset $A\subseteq \Phi$, the finitely generated abelian groups $X/\bZ A $ and $\check{X}/\bZ \check{A}$ have no $p$-torsion.   

\begin{thm}[\cite{Herpel}]  \label{t:Herpel} 
\begin{enumerate} 
\item[(i)] If $p$ is very good for $G$, then it is pretty good for $G$. 
\item[(ii)] 
Suppose $\tilde{G}$ is a possibly disconnected reductive group with connected component $G$, the prime $p=\mathrm{char}(k)$ is pretty good for $G$, and $\pi_0(\widetilde{G})$ is \'etale. Then the centraliser $C_{\tilde{G}}(H)$ of every closed subgroup scheme $H\subseteq \tilde{G}$ is smooth. 
\end{enumerate} 
\end{thm} 

\begin{cor} \label{c:etale}
If $p=\mathrm{char}(k)$ is very good for $G$, then for every $x\in T$, the centre $Z(G_x)$ is smooth. 
\end{cor} 

\begin{proof} 
Indeed, $p$ being pretty good for $G$ implies that it is pretty good for $G_x^\circ$ and that $G_x=C_G(x)$ is smooth; thus, $\pi_0(G_x)$ is \'etale. By the previous theorem,  $Z(G_x)=C_{G_x}(G_x)$ is smooth. 
\end{proof}

\section{Generic character varieties} \label{s:geometry}
Let $G$ be a connected reductive group over an algebraically closed field $k$. In this section, we define the notion generic tuples of conjugacy classes of $G$ and study the corresponding character varieties. Most of this material appears in \cite{HLRV11} for $\GL_n$ and \cite{Boalch14} for general $G$ over complex numbers.  We work with reductive groups over arbitrary fields.

\subsection{Generic classes} \label{s:generic}  Let $C_1,\ldots,C_n$ be conjugacy classes of $G$. 

\begin{defe}
The tuple $C=(C_1,\ldots, C_n)$ is called \emph{generic} if whenever there exists a proper parabolic subgroup $P\subset G$ and $x_i\in P\cap C_i$,  for $i=1,\ldots,n$, then  $\prod_{i=1}^n x_i\notin [P,P]$. 
\end{defe} 

\subsubsection{Maximal parabolics} 
Clearly it is enough to check the above condition on maximal (proper) parabolic subgroups. This implies that if $G=\GL_n$, then being generic is equivalent to requiring that whenever a proper non-zero subspace $V\subset k^n$ is preserved by $x_i\in C_i$, $i=1,\ldots,n$, we have  $\prod_{i=1}^n \det(x_i|_V)\neq 1$. Thus, we recover the notion of generic tuples of \cite[\S 2]{HLRV11}.  

\subsubsection{Standard parabolics} It is enough to check the genericity condition for \emph{standard} (proper) parabolic  subgroups. (By definition, standard parabolic subgroups are those containing a chosen Borel subgroup $B\subset G$.) Indeed, suppose $P$ is an arbitrary parabolic and choose $g \in G$ such that the conjugate $g\cdot P$ is standard. Then 
\[
x_i\in P\cap C_i\qquad \iff \qquad g\cdot x_i\in g\cdot P\cap C_i
\]
 and  
 \[
 \prod_{i=1}^n x_i\in [P,P]\qquad \iff \qquad g\cdot \prod_{i=1}^n x_i\in g\cdot [P,P]=[g\cdot P, g\cdot P].
 \] 

\subsubsection{Semisimple part}  For each conjugacy class $X$,  write $X^s$ for the corresponding semisimple class, obtained by taking the semisimple parts of the elements of $X$. 
Then one can show that $C=(C_1,\ldots,C_n)$ is generic if and only if $C^s=(C_1^s, \ldots, C_n^s)$ is generic. This is the reductive analogue of the fact that in $\GL_n$, being generic depends only on the eigenvalues.

\subsubsection{Levis instead of parabolics} For semisimple conjugacy classes, one can check genericity using  Levi subgroups instead of parabolics. More precisely,  it follows from discussions of \S \ref{s:parabolic} that a tuple of semisimple conjugacy classes  $C=(C_1,\ldots,C_n)$ is generic if and only if whenever there exists a proper standard Levi subgroup $L\subset G$ and $s_i\in L\cap C_i$,  for $i=1,\ldots,n$, then  $\prod_{i=1}^n s_i\notin [L,L]$.

\subsubsection{Generic elements} 
We now explain that deciding whether a tuple of (semisimple) conjugacy classes is generic involves checking finitely many conditions. To this end, we need a definition. 

\begin{defe} A tuple $S=(S_1,\ldots,S_n)\in T^n$ is said to be generic if for every $(w_1,\ldots, w_n) \in W^n$ and every proper standard Levi subgroup $L\subset G$, we have 
\[
\prod_{i=1}^n w_i\cdot S_i \notin T\cap [L,L]. 
\] 
\end{defe} 

Let $C_i:=G\cdot  S_i$ be the conjugacy class corresponding to $S_i$. 

\begin{lem}  The element $S=(S_1,\ldots,S_n)\in T^n$ is generic if and only if the tuple of conjugacy classes $C=(C_1,\ldots,C_n)$ is generic. 
\end{lem}

\begin{proof} 
It is clear that if $C$ is generic then so is $S$. For the converse, let $x=(x_1,\ldots,x_n) \in C_1\times\ldots \times C_n$. Let $L\subset G$ be a proper standard Levi subgroup containing $x_1,\ldots,x_n$. We need to show $\prod_{i=1}^n x_i\notin [L,L]$. We have $x_i=l_i \cdot  t_i$ for some $l_i\in L$ and $t_i\in T$. Now $t_i$ and $S_i$ are elements of $T$ which are $G$-conjugate; thus, they must actually be conjugate under $W$; i.e., there exists $w_i\in W$ such that $t_i=w_i \cdot  S_i$, cf.\ proof of Proposition 3.7.1 in \cite{Carter93}. By assumption, $t_1\cdots t_n\notin [L,L]$ which implies that $x_1\cdots x_n\notin [L,L]$. Thus, $C$ is generic.
\end{proof}

\subsubsection{Existence} The above lemma shows that generic conjugacy classes exist. Indeed, if $L$ is a proper standard levi subgroup, then $T\cap [L,L]$ is a proper closed subvariety of $T$. (At the level of Lie algebras, $\ft=\ft\cap [\fl,\fl]\oplus Z(\fl)$ for $\ft=\mathrm{Lie}(T)$ and the centre $Z(\fl)$ is non-trivial if $\fl$ is a proper Levi subalgebra of $\fg$.) Since there are only finitely many standard Levis $L$ and finitely many $w$'s, we conclude that the generic locus of $T^n$ is open dense.

\subsubsection{} \label{s:existence} Let  $\mathbf{T}_n \subseteq T^n$ denote the subvariety consisting of those tuples $(S_1,\ldots,S_n)\in T^n$ satisfying $\prod_{i=1}^n S_i \in [G,G]$. Then the above argument also shows that the generic locus of $\mathbf{T}_n$ is open dense. Since the (strongly) regular locus is also open dense in $T^n$, it follows that conjugacy classes satisfying Assumption \ref{a:mainAss} exist.

\subsection{Generic character varieties} \label{s:genericVar} In this section, we assume that $p=\mathrm{char}(k)$ is a very good prime for $G$ and $Z=Z(G)$ is connected.  
Let $C=(C_1,\ldots, C_n)$ be a generic tuple of conjugacy classes of $G$ satisfying $\prod_{i=1}^n C_i \subseteq [G,G]$. For ease of notation, let $M\colon G^{2g}\times C_1\times \cdots \times C_n \ra [G,G]$ be the map
\[
M(a_1,b_1,\ldots, a_g, b_g, c_1,\ldots, c_n) := [a_1,b_1]\cdots [a_g, b_g]c_1\cdots c_n.
  \]
The representation variety associated to $(\Gamma_{g,n}, G, C)$  is defined by $\bR:=M^{-1}(1)$. This is an affine scheme of finite type over $k$.  We assume that $R$ is non-empty (this is guaranteed if $g>0$).

\begin{thm} Every element of $\bR$ is irreducible and stable and the action of $G/Z$ on $R$ has finite \'etale stabilisers. Moreover, 
 $\bR$ is smooth and equidimensional with
\[
\dim(\bR) = 2g\dim(G)+\sum_{i=1}^n \dim(C_i) - \dim([G,G]). 
\]
\end{thm}

\begin{proof}  Let ${r=(a_1,b_1,\ldots, a_g, b_g, c_1,\ldots, c_n)\in \bR}$ and suppose $P$ is a parabolic subgroup of $G$ containing all $a_i, b_i,$ and $ c_i$'s. Then, $[P,P]$ contains $[a_1,b_1]\cdots [a_g, b_g]$; thus, $[P,P]$ also contains $\prod_{i=1}^n c_i$. As $C$ is generic, it follows that $P=G$. Thus, $r$ is irreducible and Schur's lemma (cf.\ \cite[\S 4]{Sikora}) implies that $\mathrm{Stab}_{G/Z}(r)$ is finite.  Note that this stabiliser is the same as the centraliser $C_G(H)$, where $H:=\langle a_1,b_1,\ldots, a_g, b_g, c_1,\ldots, c_n\rangle\subseteq G$. By Theorem \ref{t:Herpel}, this group is smooth. Since we already have shown that it is finite, it follows that it is finite and \'etale.

Now stability follows from the fact that whenever a connected reductive group acts on an affine scheme of finite type with finite stabilisers, then every point is stable, cf.\ \cite[Proposition 0.8]{Mumford}. (Alternatively, one can use the Hilbert--Mumford criterion as in \cite{Boalch14}.) The statements about $\bR$ follow from the fact that $1$ is a regular value for $M$, see \cite[\S 2]{KNP25} for details.  
\end{proof} 

\subsubsection{} 
Next, let $\bX:=\bR/\!\!/G$ and $[\bX]:=[\bR/(G/Z)]$ denote, respectively, the (generic) character variety and character stack associated to $(\Gamma_{g,n}, G, C)$. The previous theorem implies: 

\begin{prop} 
$[\bX]$ is a smooth Deligne-Mumford stack with coarse moduli space $\bX$.
\end{prop}

\subsection{Frobenius Mass Formula} 
We now work over $k=\overline{\mathbb{F}_q}$ and assume that $G$ is equipped with an untwisted Frobenius automorphism $F: G\ra G$. 
We also assume that the conjugacy classes $C=(C_1,\ldots,C_n)$ are $F$-stable. Then the action of the Frobenius $F$ on $G$ and $C_i$'s, gives an action of $F$ on $\bR$, $\bX$ and $[\bX]$. We can therefore speak about the Frobenius fixed points $\bR^F$, $\bX^F$, and $[\bX]^F$. The first two are finite sets while the last one is a finite groupoid. 
By definition, 
\[
\bR^F=\Big\{(a_1,\ldots, b_g, c_1,\ldots,c_n)\in (G^F)^{2g}\times \prod_{i=1}^n C_i^F \, \Big|\, M(a_1,\ldots,b_g, c_1,\ldots,c_n)=1\Big\}.
\]

\subsubsection{Stable conjugacy} 

\begin{defe} A conjugacy class $C\subseteq G$ is called \emph{stable} if the stabiliser in $G$ of each element of $C$ is connected. 
\end{defe} 

If $C$ is stable, then  $C^F$ is a unique $G^F$-conjugacy class, cf.\ \cite[\S2.7.1]{GM20}. Thus, it makes sense to evaluate irreducible complex characters of $G^F$ on $C_i^F$.

\subsubsection{Frobenius Mass Formula} 
The following result goes back to the work of Frobenius; see e.g.  \cite[\S 3]{HLRV11} for a proof: 

\begin{thm} \label{t:Frob} If $C_i$'s are stable, then 
\[
{|\bR^F|} = |G^F|\sum_{\chi\in \Irr(G^F)} \left(\frac{|G^F|}{\chi(1)}\right)^{2g-2} \prod_{i=1}^n\frac{\chi(C_i^F)|C_i^F|}{\chi(1)}.
\] 
\end{thm}

\subsubsection{Points of the character stack} An application of Lang's theorem implies that 
\[
|[\bX]^F|=\frac{|\bR^F|}{|(G/Z)^F|},
\]
cf.\ \cite[Lemma 3.5.6]{Behrend91}. 
Thus, the Frobenius Mass Formula also provides an expression for the number of points of the character stack.

\subsubsection{Points of character variety} 
We now assume that the tuple $(C_1,\ldots,C_n)$ is generic. We saw above that $\bX$ is a coarse moduli space for the smooth Deligne-Mumford stack $[\bX]$. This implies that $[\bX]$ and $X$ have the same number of points over finite fields, cf.\ \cite[\S 2]{Behrend91} or \cite{BogaartEdixhoven}. Thus, we conclude:

\begin{cor} If $C_i$'s are generic and stable, then  
\[
|\bX^F|=|[\bX]^F| = |Z^F|\sum_{\chi\in \Irr(G^F)} \left(\frac{|G^F|}{\chi(1)}\right)^{2g-2} \prod_{i=1}^n\frac{\chi(C_i^F)|C_i^F|}{\chi(1)}.
\]
\end{cor}

\subsection{Additive character varieties} \label{ss:additive} 
The above considerations have additive analogues which we leave to the reader to formulate. The only non-trivial modification is the additive analogue of the Frobenius Mass Formula, which we now discuss. 

\subsubsection{} Let $\cO_1,\ldots,\cO_n$ be $F$-stable adjoint $G$-orbits in $\fg$. Let $A$ denote the corresponding additive representation variety. For ease of notation, let $m\colon \fg^{2g}\times \cO_1\times \cdots \times \cO_n\rightarrow [\fg,\fg]$ be the map
\[
m(a_1,b_1,\ldots,a_g, b_g, c_1,\ldots,c_n):=\sum_{i=1}^g [a_i,b_i] +\sum_{i=1}^n c_i. 
\]
Then, we have
\[
\bA^F=\Big\{(a_1,\ldots, b_g, c_1,\ldots,c_n)\in (\fg^F)^{2g}\times \prod_{i=1}^n \cO_i^F\, \Big|\, m(a_1,\ldots,b_g, c_1,\ldots,c_n)=0\Big\}.
\]

 \subsubsection{} Let $1_{\cO_i}^G$ denote the characteristic function of the $G^F$-orbit $\cO_i^F$. Let $ \mathcal{F}$ denote the Fourier transform on the space of $G^F$-invariant functions on $\fg^F$ (see \S \ref{s:invariant} for the definition).  For each $x\in \fg$, let $\fg_x$ denote its centraliser in $\fg$.  Then we have (cf. \cite[\S 3]{HLRV11}):

\begin{thm}[Additive Frobenius Mass Formula] \label{t:additiveFrob} If $\cO_i$'s are stable, then 
\[
{|\bA^F|} = |\fg^F|^{g-1} \sum_{x\in \fg^F} |\fg_x^F|^g \prod_{i=1}^n \mathcal{F}(1_{\cO_i}^G)(x).
\] 
\end{thm}

\begin{cor} \label{c:countY} 
If $\cO_i$'s are generic and stable, then
\[
|Y^F|=|[Y]^F|=\frac{|\bA^F||Z^F|}{|G^F|}=\frac{|Z^F||\fg^F|^{g-1}}{|G^F|} \sum_{x\in \fg^F} |\fg_x^F|^g \prod_{i=1}^n \mathcal{F}(1_{\cO_i}^G)(x).
\]
\end{cor}

\section{Counting points on character varieties} \label{s:countX} 
In this section, we count points on the character varieties introduced in \S \ref{s:main} and prove Theorem \ref{t:countX}.   
We treat the case of a once-punctured surface first and mention the modifications required in the multi-punctured case at the end.  

Let $G$ be a connected split reductive group over $k=\overline{\Fq}$ equipped with an untwisted Frobenius $F$, giving $G$ an $\Fq$-structure.  We assume $Z(G)$ is connected and $p=\mathrm{char}(k)$ is very good for $G$. 
Let $T$ be a maximal split $F$-stable torus of $G$. Let $S\in T^F$ be a strongly regular generic element and $C=G\cdot S$ the corresponding conjugacy class in $G$. Let $X$ denote the character variety associated to the once-punctured Riemann surface and the conjugacy class $C$.

By the Frobenius Mass Formula (Theorem \ref{t:Frob}), we have 

\begin{equation}\label{eq:Frob}
|X^F|=\frac{|\bR^F|}{|(G/Z)^F|} = \frac{|Z^F|}{|T^F|}\sum_{\chi\in \Irr(G^F)} \left(\frac{|G^F|}{\chi(1)}\right)^{2g-1} \chi(S). 
\end{equation} 

Note that the equality continues to hold if $F$ is replaced by $F^m$ for a positive integer $m$. 
In what follows, we use a crucial result of Deligne and Lusztig to rewrite the above sum in a more tractable form. 

\subsection{A theorem of Deligne and Lusztig} 
Given a finite abelian group $A$, write $A^\vee$ for the Pontryagin dual $\Hom(A, \bCt)$. 
Recall that a principal series representation is an irreducible constituent of 
\[
R_T^G\, \theta:=\mathrm{Ind}_{B^F}^{G^F}\, \theta,\qquad \theta \in (T^F)^\vee. 
\]

 \begin{thm}[Corollary 7.6 of \cite{DL76}] Let $\chi\in \Irr(G^F)$ be an arbitrary character and $S\in T^F$ a strongly regular element. Then 
 \[
 \chi(S) = \sum_{\theta \in (T^F)^\vee} \langle \chi, R_T^G\, \theta \rangle\, \theta(S). 
 \]
 \label{t:DL}
\end{thm} 

\begin{cor} \label{c:regular} 
If $S\in T^F$ is a strongly regular element and $\chi(S)\neq 0$, then $\chi$ is a principal series representation.  
\end{cor}

The above corollary implies that only principal series representations contribute to the sum in the expression \eqref{eq:Frob} for $|X^F|$.
Note that one does not need Deligne--Lusztig theory to formulate the above theorem and its corollary. However, the only proof we know of uses the full power of the Deligne--Lusztig theory.

\subsubsection{}  As we have assumed that $G$ has connected centre, the induced representations $R_T^G\, \theta$ and $R_T^G\, \theta'$ have a common constituent if and only if $\theta$ and $\theta'$ are $W$-conjugate, cf.\ \cite[Corollary 6.3]{DL76}.  
Thus, we can associate to a principal series representation $\chi\in \Irr(G^F)$ a character $\theta=\theta_\chi \in (T^F)^\vee$, well-defined up to $W$-conjugacy.

\subsubsection{Interlude on Pontryagin duality} \label{s:Pontryagin}
Let $\mu_{\infty,p'}(\bC)$ denote the set of  roots of unity in $\bCt$ whose order is prime to $p$. We choose, once and for all, group isomorphisms
\begin{equation}\label{eq:rootsOfUnity}
k^\times \simeq (\mathbb{Q}/\mathbb{Z})_{p'} \qquad \textrm{and} \qquad (\mathbb{Q}/\mathbb{Z})_{p'}\simeq \mu_{\infty,p'}(\bC). 
\end{equation} 
 As noted in \cite[\S 5]{DL76}, this induces a group isomorphism 
 \begin{equation}
(T^F)^\vee\simeq \chT^F
\end{equation}

\subsubsection{} Using the above identification, we can think of $\theta\in (T^F)^\vee$ as an element of $\check{T}^F\subseteq \check{G}^F$. 
Let $\check{G}_\theta$ denote the centraliser of $\theta$ in $\check{G}$. The fact that $Z(G)$ is connected implies that $\chG$ has simply connected derived subgroup. By Steinberg's theorem, the stabiliser  $\chG_\theta$ is a \emph{connected} standard pseudo-Levi subgroup of $\chG$.  The Weyl group of $\chG_\theta$ equals 
\[
W_\theta=\{w\in W\, |\, w\cdot \theta=\theta\}.
\] 
The Langlands dual group $G_\theta:=(\check{G}_\theta)^{\check{}}$ is a standard endoscopy group for $G$ with Weyl group  $W(G_\theta)=W_\theta$.

\subsubsection{} \label{s:rep} Since $G$ is assumed to have connected centre, the constituents of $R_T^G \, \theta$ are in canonical bijection with $\Irr(W_\theta)$,  cf.\ \cite[Corollary 4.20]{Kilmoyer78} or \cite[Theorem 6.8]{DL76}.  Let $\chi_{\theta, \rho} \in R_T^G\, \theta$ be the irreducible character corresponding to $\rho\in \Irr(W_\theta)$. Then, the double centraliser theorem (cf.\ \cite[Theorem 5.18.1]{Etingof11}) gives
\[
\langle \chi_{\theta, \rho}, R_T^G\, \theta\rangle = \dim(\rho).
\]

\subsubsection{} Consider the set of pairs $(L, \rho)$, where $L$ is a standard endoscopy group for $G$ and $\rho$ is an irreducible character of $W(L)$. Note $W$ acts on this set by conjugation. Let $\cT(G)$ denote the set of $W$-orbits. We call the elements of $\cT(G)$ \emph{$G$-types} and denote them by $\tau=[L,\rho]$. Note that $\cT(G)$ is a finite set which depends only on the root system $\Phi$ of $G$; in particular, it is independent of the ground field.

By the above discussions, we can associate a type to every principal series representation. (This is a reductive analogue of the notion of type used in \cite{HRV08, HLRV11, Cambo17}.) In other words, we have a map 
\begin{equation} \label{eq:typeMap}
\textrm{Principal series representations of $G^F$} \ra \cT(G). 
\end{equation} 

\subsubsection{Unipotent characters} Now let $L$ be a standard endoscopy subgroup of $G$. The irreducible constituents of $R_T^L 1$ (which are known as unipotent principal series representations of $L^F$) are in bijection with irreducible characters of $W(L)$. Let $\tilde{\rho}$ denote the unipotent principal series character of $L^F$ corresponding to $\rho\in \Irr(W(L))$. 

 \begin{defe} We define the \emph{$q$-mass} of $\tau=[L,\rho]\in \cT(G)$ by 
 \[
 m_\tau(q):= q^{|\Phi^+(G)| - |\Phi^+(L)|} \frac{|L^F|}{\tilde{\rho}(1)}.
 \]
 \end{defe} 
 Here, $\Phi^+(L)$ denotes the set of positive roots of $L$. 
 Note that $\tilde{\rho}(1)$ is a polynomial in $q$ which divides $|L^F|$, cf. \cite[Remark 2.3.27]{GM20}.
Thus, $m_\tau(q)$ is a polynomial in $q$. 

\subsubsection{} The relevance of $m_\tau(q)$ emerges in the following proposition which can be found in, for instance,  \cite[Corollary 2.6.6]{GM20}: 

  \begin{prop} \label{p:mtau}
  Suppose $\chi$ is a principal series character of $G^F$ of type $\tau=[L,\rho]$. Then 
\[
\displaystyle   \frac{|G^F|}{\chi(1)} = m_\tau(q). 
\] 
\end{prop}

\begin{rem} \label{r:divide} As noted in \cite[Remark 2.3.7]{GM20}, for all irreducible characters $\chi$, we have that
$|T^F|$ divides $m_\tau(q)$ (as polynomials in $q$).
\end{rem} 
   
 \subsubsection{} We can now rewrite the Frobenius Mass Formula \eqref{eq:Frob} in terms of types. To this end, let  $\Irr(G^F)_\tau$ denote the set of irreducible principal series characters of type $\tau$. 
Then, we have
 \[
|\bX^F| = \frac{|Z^F|}{|T^F|} \sum_{\tau \in \cT(G)} m_\tau(q)^{2g-1}  \sum_{\chi \in \Irr(G^F)_\tau} \chi(S). 
\] 

\subsubsection{} \label{s:Stau} For ease of notation, define   
\[
\boxed{
S_\tau(q):= \displaystyle \sum_{\chi \in \Irr(G^F)_\tau} \chi(S).}
\]
Then, we find
 \[
|\bX^F| = \frac{|Z^F|}{|T^F|} \sum_{\tau\in \cT(G)} m_\tau(q)^{2g-1} S_\tau(q).
\] 
The importance of the above formula is that the sum is over a parameter which does not depend on $q$. Moreover, it implies that computing $|\bX^F|$ is reduced to computing the character sums $S_\tau(q)$. 
\begin{rem} 
\begin{enumerate} 
\item[(i)] In obtaining the above formula, we used the assumption that $S$ is strongly regular, but not the assumption that $S$ is generic.
\item[(ii)] The above formula holds if $F$ is replaced by $F^m$ for a positive integer $m$. 
\end{enumerate} 
\end{rem} 

\subsection{Evaluating the character sum $S_\tau$}\label{s:onepunctureStau}   In this subsection, we rewrite the expression for $S_\tau$ in terms of sums of characters of $T^F$, thus turning the problem of computing $S_\tau$ into a ``commutative" problem. 

Suppose $\tau=[L,\rho]\in \cT(G)$  where $L$ is a standard endoscopy group for $G$ with Weyl group $W(L)$ and $\rho$ is an irreducible character of $W(L)$. Let $[L]$ denote the $W$-orbit of $L$ in $G$ and  $[W(L)]$ the $W$-orbit of $W(L)$ in $W$. 

\subsubsection{}\label{s:taufibre} We start by elucidating the nature of $\Irr(G^F)_\tau$. 
It  consists of characters $\chi_{\theta,\rho}$ as $\theta$ runs over elements of $(T^F)^\vee$ satisfying $[G_\theta]=[L]$. Equivalently, 
\[
\Irr(G^F)_\tau=\{\chi_{\theta, \rho}\, |\, \theta \in (T^F)^\vee, \,  [W_\theta]=[W(L)]\}. 
\]
We refer the reader to \S \ref{s:rep} for the definition of $\chi_{\theta, \rho}$.

\subsubsection{} The above discussion implies
\begin{multline*}
S_\tau(q) = \displaystyle \sum_{\chi \in \Irr(G^F)_\tau} \chi(S) =
\frac{|W(L)|}{|W|}\sum_{\substack{\theta\in \chT^F \\ [W_\theta]=  [W(L)]}} \chi_{\theta, \rho}(S) =
\frac{|W(L)|}{|N_W(W(L))|}\sum_{\substack{\theta\in \chT^F \\ W_\theta=  W(L)}} \chi_{\theta, \rho}(S).
\end{multline*}
The second equality follows from the fact that $\chi_{\theta, \rho}$ depends only on the $W$-orbit of $\theta$.

\subsubsection{} \label{s:Carter} Recall that $[L]$ denotes the orbit of $L$ under the action of $W$ on subgroups of $G$. According to \cite[Lemma 34]{Carter72}, we have ${|W/N_W(W(L))|}=|[L]|$. Thus, we obtain 
\[
S_\tau(q) =
\frac{|W(L)||[L]|}{|W|}\sum_{\substack{\theta\in \chT^F \\ W_\theta=  W(L)}} \chi_{\theta, \rho}(S).
\]

\begin{rem} 
In \cite[Proposition 28]{Carter72}, one finds the following formula for $|[L]|$. Let $W^\perp(L)$ be the Weyl group generated by the roots orthogonal to $\Phi(L)$.  Let $A_W(L)$ denote the group of symmetries of the Dynkin diagram of $L$ induced by elements of $W$. Then 
\[
|[L]|=\frac{|W|}{|W(L)||W^\perp(L)||A_W(L)|}. 
\]
\end{rem} 

\subsubsection{} Next, Theorem \ref{t:DL} implies 
\[
\displaystyle \chi_{\theta,\rho}(S) =   \sum_{\theta'\in \check{T}^F} \langle \chi_{\theta, \rho}, R_{T}^G \, \theta'\rangle \,\theta'(S).
\]
Note that a summand is zero unless $\theta'$ is in the same $W$-orbit as $\theta$, in which case it  equals $\dim(\rho)\, \theta'(S)$; thus, 
\[
\displaystyle \chi_{\theta,\rho}(S)  =
\dim(\rho) \sum_{w\in W/W_\theta}   (w\cdot \theta)(S)  = \frac{\dim(\rho)}{|W_\theta|} \sum_{w\in W}   \theta(w\cdot S).  
\]
Note that the last sum depends only on the $W$-orbit of $\theta$.

\subsubsection{} We now plug in the explicit formula for $\chi_{\theta, \rho}$ into the formula for $S_\tau$ to obtain
\[
S_\tau(q) = 
\frac{|W(L)||[L]|}{|W|} \sum_{\substack{\theta\in \chT^F \\ W_\theta = W(L)}} \frac{\dim(\rho)}{|W_\theta|} \sum_{w\in W}   \theta(w\cdot S)
=\frac{\dim(\rho)|[L]|}{|W|}\sum_{w\in W}  \sum_{\substack{\theta\in \chT^F \\ W_\theta = W(L)}}     \theta(w\cdot S).
\]

\subsubsection{} To alleviate notation, let  
\[
\boxed{
\alpha_{L,S}(q):= \sum_{\substack{\theta\in \chT^F \\ W_\theta=W(L)}} \theta(S).} 
\]
Thus, we have shown 
\[
S_\tau(q)= \frac{\dim(\rho)|[L]|}{|W|} \sum_{w\in W} \alpha_{L,w\cdot S}(q).
\]
This formula continues to hold if $q$ is replaced by $q^m$ for a positive integer $m$. 

The above equality implies that the problem of computing $S_\tau$ reduces to evaluating the character sum $\alpha_{L,S}$. To compute the latter, we first consider an auxiliary sum $\alpha_{L,S}^\supseteq$ which is easier to handle.

\subsection{The auxiliary sum $\alpha_{L,S}^\supseteq$} \label{ss:Frob}
 Given a standard endoscopy group $L$ for $G$ and $S\in T^F$, define 
\[
\alpha_{L,S}^\supseteq(q):=\sum_{\substack{\theta\in \chT^F \\ W_\theta  \supseteq W(L)}} \theta(S)
=\sum_{\theta \in (\check{T}^F)^{W(L)}} \theta(S).
\]

\begin{lem} We have
\[
\alpha^{\supseteq}_{ L ,S}(q)= 
\begin{cases} |(\chT^F)^{W(L)}| & \textrm{if $S\in [L^F, L^F]$}\\
0 & \textrm{otherwise}. 
\end{cases} 
\]
This holds if $q$ and $F$ are replaced by $q^m$ and $F^m$, for any positive integer $m$.  
\end{lem} 
\begin{proof} The lemma follows from the fact (proved in \cite[\S 5.2]{KNP25}) that the Pontryagin dual of the embedding of finite abelian groups $(\chT^F)^{W(L)} \hookrightarrow \chT^F$ is the canonical quotient map 
\[
T^F \twoheadrightarrow \frac{T^F}{T^F\cap [L^F, L^F]}.\qedhere
\]
%using properties of $W(L)$-invariant characters of $T^F$; see, for instance, \cite{kamgarpour2012ramified}, Proposition 23.
\end{proof} 

\subsubsection{} Now we simplify the above expression using the fact $S\in [G,G]$ is assumed to be generic. This implies that $S\in [L,L]$ if  and only if the root system $\Phi(L)$ is isolated in $\Phi=\Phi(G)$; i.e., $\Phi(L)$ is not contained in any proper Levi subsystem of $\Phi$.

 Note that the containment $[L^F, L^F] \subseteq [L,L]$ may be proper (e.g. $L=G=\PGL_n$). Thus, it could happen that $S\in [L,L]\setminus [L^F, L^F]$. However, this subtlety goes away after a finite base change. Indeed, 
 there are finitely many isolated standard endoscopic groups for $G$; thus, given $S\in T^F$, there exists a positive integer $m$, such that  $S\in [L^{F^m}, L^{F^m}]$ for all isolated standard endoscopy groups. We conclude: 

\begin{cor} There exists a positive integer $m$ such that 
\[
\alpha^{\supseteq}_{ L ,S}(q^m)= 
\begin{cases} |(\chT^{F^m})^{W(L)}| & \textrm{if $L$ is isolated}\\
0 & \textrm{otherwise}. 
\end{cases} 
\]
Moreover, the same formula holds if $m$ is replaced by $mx$ for a positive integer $x$. 
\end{cor} 

\subsubsection{Description of $\check{T}^{W(L)}$} To proceed further, we need a description of $T^{W(L)}$. It is proved in \cite[\S 5]{KNP25} that, under the assumption that $G$ has connected centre, we have:
\[
\check{T}^{W(L)} \simeq Z(\check{L}). 
\]
Thus, 
\[
 |(\chT^F)^{W(L)}|=|Z(\check{L})^F|=|\pi_0(Z(\check{L}))^F|\times |(Z(\check{L})^\circ)^F|.
 \]
By Corollary \ref{c:etale}, the prime $p=\mathrm{char}(k)$ does not divide $|\pi_0(Z(\check{L}))|$. Thus, there exists a positive integer $n$ such that  $q^n\equiv 1 \mod |\pi_0(Z(\check{L}))|$. This, in turn, implies 
\[
\pi_0(Z(\check{L}))^{F^n}= \pi_0(Z(\check{L})).
\]
Combining this with the previous corollary, we obtain: 

\begin{cor} \label{c:r} 
There exists a positive integer $r$ such that for every standard endoscopy $L$ of $G$, we have 
\[
\alpha^{\supseteq}_{ L ,S}(q^r)= 
\begin{cases} |\pi_0(Z(\check{L}))|\times |(Z(\check{L})^\circ)^{F^r}|  & \textrm{if $L$ is isolated}\\
0 & \textrm{otherwise}. 
\end{cases} 
\]
The same formula holds if $r$ is replaced by $rx$ for any positive integer $x$. \end{cor}

\subsection{Evaluating $\alpha_{L,S}$}\label{s:nu} We now give an expression for $\alpha_{L,S}$ using M\"{o}bius inversion. 

\subsubsection{} 
Let $\mu$ denote the M\"{o}bius function on the (finite) poset of the standard endoscopy groups of $G$, ordered by inclusion of their root systems (Note that these root systems all live in $\Phi(G)$). M\"{o}bius inversion implies 
\[
\alpha_{L,S} = \sum_{L'} \mu(L,L') \alpha^{\supseteq}_{L',S}, 
\]
where  the sum is over standard endoscopy groups $L'$ satisfying $\Phi(L')\supseteq \Phi(L)$.  
 \subsubsection{} We have seen that $\alpha_{L',S}^\supseteq$ is $0$ unless $L'$ is isolated. In this case, our assumption that $G$ has connected centre implies  
 \[
 Z(\check{L'})^\circ = Z(\check{G})^\circ \simeq Z(G)^\circ= Z(G). 
 \]

 \subsubsection{} To simplify notation, let 
\[
\nu(L):= \sum_{L'} |\pi_0(Z(\check{L'}))| \,  \mu(L, L').
\]
Here, the sum runs over isolated standard endoscopy groups $L'$ of $G$ whose root systems contain $\Phi(L)$.   Thus, $\nu(L)$ is an integer depending only on the root data of $L$ and $G$; in particular, it is independent of the ground field $k$ (of very good characteristic). In fact, $\nu(L)$ depends only on the $W$-orbit of $L$. In view of the above discussions, we conclude:

\begin{prop}\label{p:alpha}  There exists a positive integer $r$ such that 
\[
\alpha_{L,S}(q^r) =  |Z^{F^r}|\,  \nu(L).
\]
The equality continues to hold if $r$ is replaced by $rx$ for any positive integer $x$.
\end{prop} 

Observe that $\alpha_{L,S}$ is independent of $S$ (provided $S$ is generic).

\subsection{Proof of Theorem \ref{t:countX} in the once-punctured case} 
The above discussions imply that there exists a positive integer $r$ such that 
\[
|\bX(\mathbb{F}_{q^{r}})| =\frac{|Z(\mathbb{F}_{q^{r}})|}{|T(\mathbb{F}_{q^{r}})|} \sum_{\tau=[L, \rho]\in \cT(G)} S_\tau(q^{r}) \,m_\tau(q^{r})^{2g-1}, 
\]
where 
\[
S_\tau(q^{r}) =  \frac{\dim(\rho)\, |[L]|}{|W|} \sum_{w\in W} \alpha_{L,w\cdot S}(q^{r}) =  \dim(\rho) \, |[L]| \, \nu(L)\,  |Z(\mathbb{F}_{q^{r}})|.  
\]
Moreover, the above formulas continue to hold if $r$ is replaced by a positive multiple. 
\subsubsection{}\label{sss:countingpoly} To proceed further, we need some notation. 
Let $f_Z$,  $f_T$, $f_L$ be counting polynomial for $Z$, $T$, $L$. Thus,
\[
f_Z(t)= (t-1)^{\dim(Z)},\quad f_T(t)=(t-1)^{\dim(T)},\quad f_L(t)=f_T(t) t^{|\Phi^+(L)|} P_L(t), 
\]
where $P_{L}(t)\in \mathbb{Z}[t]$ is the Poincar\'e polynomial of the flag variety of $L$. 
Recall that for a type $\tau=[L,\rho]$, the polynomial $m_\tau(t)$ is defined by \[
 m_\tau(t):= t^{|\Phi^+(G)| - |\Phi^+(L)|} \frac{|f_L(t)|}{\tilde{\rho}(1)}\in \mathbb{Q}[t]. 
 \]
 
 \begin{defe} 
Define the rational function 
\[
f_X(t):=\frac{f_Z(t)^2}{f_T(t)} \sum_{\tau=[L, \rho]\in \cT(G)} \dim(\rho) |[L]| \,m_\tau(t)^{2g-1}\in \mathbb{Q}(t). 
\]
\end{defe} 

\subsubsection{Conclusion of the proof} We have seen that there exists a positive integer $r$ such that  
\[
|X(\mathbb{F}_{q^{rm}})| = f_X(q^{rm}),\qquad \forall \, m\in \mathbb{Z}_{>0}. 
\]
Thus, $f_X(q^{rm})$ is an integer for all $m>0$. An elementary argument (cf.  \cite[Lemma 2.11]{GoodwinRohrle}) then implies that $f_X(t)$ is a polynomial. Thus, $X\times_{\Fq} \mathbb{F}_{q^r}$ is polynomial count with counting polynomial $f_X$. This concludes the proof of Theorem \ref{t:countX} in the single punctured case.  
\qedhere

\subsection{Proof of Theorem \ref{t:countX} in the multi-punctured case} \label{s:countXPrecise} Let $\bX$ be a character variety associated to an $n$-punctured surface group satisfying Assumption \ref{a:mainAss}. Define the rational function 
\begin{equation} \label{eq:f_X}
f_X(t):=\frac{|W|^{n-1} f_Z(t)^2}{f_T(t)^n} \sum_{\tau=[L, \rho]\in \cT(G)} \frac{\dim(\rho)^n  |[L]| \,m_\tau(t)^{2g-2+n}}{|W(L)|^{n-1}}\in \mathbb{Q}(t). 
\end{equation} 
Our goal is to show that  
\[
|X(\mathbb{F}_{q^{rm}})| = f_X(q^{rm}),\qquad \forall \, m\in \mathbb{Z}_{>0}. 
\]

\subsubsection{} 
As in the once-punctured case, we have
\[
|X^F| = \frac{|Z^F|}{|T^F|^n} \sum_{\tau=[L,\rho]\in\cT(G)} m_\tau(q)^{2g-2+n} S_\tau(q),
\]
where $S_\tau$ is now defined as
\[
S_\tau(q) := \sum_{\chi\in\Irr(G^F)_\tau} \chi(S_1)\cdots\chi(S_n).
\]

\subsubsection{} Proceeding as in the once-punctured case, we compute
\[
S_\tau(q) = \frac{|W(L)|}{|N_W(W(L))|} 
\sum_{\substack{\theta\in \check{T}^F \\ W_\theta = W(L)}} 
\prod_{i=1}^n \chi_{\theta,\rho}(S_i)
= \frac{|W(L)|}{|W|} |[L]| \sum_{\substack{\theta\in \check{T}^F \\ W_\theta = W(L)}} 
\prod_{i=1}^n \chi_{\theta,\rho}(S_i).
\]
We already computed 
\[
\chi_{\theta,\rho}(S_i) = \frac{\dim(\rho)}{|W_\theta|} \sum_{w\in W} \theta(w\cdot S_i).
\]
Plugging this expression for $\chi_{\theta,\rho}(S_i)$ into the one for $S_\tau$ yields
\[
S_\tau(q) = \frac{\dim(\rho)^n\, |[L]|}{|W|\, |W(L)|^{n-1}} \sum_{\underline{w} \in W^n} \alpha_{L,\underline{w}\cdot\underline{S}}(q),
\]
where $\underline{w} := (w_1,\ldots,w_n)\in W^n$ and $\underline{w}\cdot\underline{S} := (w_1\cdot S_1)\cdots(w_n\cdot S_n)\in T^F$. 

\subsubsection{} We have seen that there exists a positive integer $r$ such that 
\[
S_\tau(q^r) = \frac{|W|^{n-1}}{|W(L)|^{n-1}} \dim(\rho)^n\, |[L]|\, \nu(L)\, |Z^{F^r}|.
\]
This means we have
\begin{equation}\label{eq:mainFormula}
|X^{F^r}| = \frac{|W|^{n-1}|Z^{F^r}|^2}{|T^{F^r}|^n} 
\sum_{\tau=[L,\rho]\in\cT(G)} 
\frac{\dim(\rho)^n\, |[L]|\, m_\tau(q^r)^{2g-2+n}}{|W(L)|^{n-1}} 
= f_X(q^r).
\end{equation}
This continues to hold if $r$ is replaced by $rm$ for any positive integer $m$. The same reasoning as in the single punctured case completes the proof of Theorem \ref{t:countX}.
\qed

\section{Topological implications} \label{s:topology} The aim of this section is discuss the topological implications of our expression for the counting polynomial of generic character varieties and prove Corollary \ref{c:X}. 
We continue using the notation of the previous section.

\subsection{Palindromic property} \label{s:palindromic} We start by recalling some facts about Alvis--Curtis duality; see \cite[\S 7.2]{DM20} for further details. 

\subsubsection{} Alvis--Curtis duality is (up to a sign) an involution 
\[
\mathfrak{D}: \Irr(G^F)\ra \Irr(G^F).
\]
To compute the dimension of the irreducible character $\mathfrak{D}(\chi)$, we need to replace $t$ by $t^{-1}$ in the polynomial $d_\chi(t)\in \mathbb{Q}[t]$ encoding the dimension of $\chi$. More precisely, 
\[
\dim \mathfrak{D}(\chi)=t^{|\Phi^+|} d_\chi (t^{-1}). 
\]

\subsubsection{} 
It is known that $\mathfrak{D}$ restricts to an involution on the set of irreducible constituents of $R_T^G\, \theta$; namely, it sends the irreducible constituent corresponding to $\rho \in \Irr(W_\theta)$ to the constituent corresponding to $\epsilon \rho$, where $\epsilon$ is the sign character of $W_\theta$. In particular, we see that $\mathfrak{D}$ sends an irreducible character of type $[L, \rho]$ to one of type $[L, \epsilon \rho]$.

\subsubsection{} 
Now, observe that 
\[
f_Z(t^{-1}) = (-1)^{\dim(Z)} t^{-\dim(Z)} f_Z(t),\qquad f_T(t^{-1}) = (-1)^{\dim(T)} t^{-\dim(T)} f_T(t),
\]
and
\[
f_L(t^{-1}) = (-1)^{\dim(T)} t^{-\dim(T)-3|\Phi^+(L)|} f_L(t).
\]
Moreover, we have 
\[
m_\tau(t^{-1}) = (-1)^{\dim(T)} t^{-\dim(T)-|\Phi(G)|} m_{\tau'}(t),
\]
where $\tau' = [L,\epsilon\rho]$. Thus, we conclude: 
\[
f_X(t^{-1}) = t^{-\deg(f_X)} \frac{|W|^{n-1}f_Z(t)^2}{f_T(t)^n} \sum_{\tau=[L,\rho]\in\cT(G)} \frac{
\dim(\rho)^n  |[L]| \,m_{\tau'}(t)^{2g-2+n}
}{
|W(L)|^{n-1}
}.
\]
Since the map $\tau=[L,\rho]\mapsto \tau'=[L,\epsilon\rho]$ is an involution, we conclude $t^{\deg(f_X)}f_X(t^{-1})$ equals $f_X(t)$; i.e., $f_X$ is a palindrome. \qed

\subsection{Number of connected components}  
The approach here is very similar to \cite[Corollary 5]{BK22}, where it is shown that the number of components of the character stack associated to a \emph{compact} surface group is $|\pi_0(Z(\check{G}))|$.  Thus, we only sketch the main idea.  

Since $X$ is smooth and equidimensional, connected components and irreducible components coincide. Thus, the number of connected components equals the leading coefficient of the $E$-polynomial. By the above discussions, the $E$-polynomial of $\overline{X}$ equals $f_X$. Thus, our goal is to determine the leading term of $f_X$. 

Following the same analysis as \cite[Corollary 5]{BK22}, one can show that  the only type contributing to the leading term of $f_X$ is the $[G, 1]$. This is the type corresponding to one-dimensional (principal series) representations of $G(\Fqr)$; i.e. representations of the form 
\[
G(\Fqr) \ra G(\Fqr)^{\mathrm{ab}} \ra \bCt.
\]
The number of such representations is 
\[
|G(\Fqr)^{\mathrm{ab}}| = |Z(\check{G}(\Fqr))| = |\pi_0(\check{Z}(G))(\Fqr)|\times q^{r\dim(Z)}. 
\]
As $q$ is  coprime to $|\pi_0(\check{Z}(G))|$ and $r$ is chosen so that $|\pi_0(\check{Z}(G))(\Fqr)|=|\pi_0(\check{Z}(G))|$,  we conclude that the leading coefficient of $f_X$ is 
\[
|\pi_0(Z(\check{G}))(\Fqr)|=|\pi_0(Z(\check{G}))|. 
\]
\newpage

\subsection{Euler characteristic} \label{s:EulerX} Note that $E(\overline{X})=f_X(1)$. 
We have three cases: 

\subsubsection{Genus $g=0$}  The formula \ref{eq:mainFormula} implies 
\[
f_X(1) = \frac{1}{(2r)!} \frac{d^{2r}}{dq^{2r}} \bigg|_{q=1} \xi(q),
\]
where $r:=\dim(T)-\dim(Z)$ is the semisimple rank of $G$ and $\xi$ is the rational function
\[
\xi(q) := q^{|\Phi^+|(n-2)} \sum_{L} \bigg(\frac{|W|}{|W(L)|}\bigg)^{n-1} \nu(L) \sum_{\rho} \dim(\rho)^n \bigg(\frac{P_{W(L)}(q)}{\dim(\tilde{\rho})}\bigg)^{n-2}.
\]
Here, the first sum is over all standard endoscopy groups of $G$, and the second sum is over all irreducible characters of $W(L)$.

\subsubsection{Genus $g=1$ and trivial centre}
The formula \eqref{eq:mainFormula} implies
\[
f_X(1) = |W|^{n-1} \sum_{L} |W(L)|\, |\Irr(W(L))|\, \nu(L),
\]
where the sum is over all standard endoscopy groups of $G$.

\subsubsection{Genus $g>1$ or $g=1$ and non-trivial centre} We shall prove that $f_X(1)=0$, thus establishing Corollary \ref{c:X} (c). By Remark \ref{r:divide}, the polynomial $f_T(t)$ divides $m_\tau(t)$. This implies that if $g>1$, then $f_T$ divides $f_X$; hence $f_X(1)=0$. 
It also implies that if $g=1$, then $f_Z$ divides $f_X$. Since $Z$ is non-trivial and connected, it follows that $f_Z(1)=0$. Thus, $f_X(1)=0$.

\section{Counting points on additive character varieties}\label{s:countingY} \label{s:countY} In this section, we count points on the additive character variety introduced in \S \ref{s:additive} and prove Theorem \ref{t:countY} and Corollary \ref{c:Y}.  
For the ease of notation, we  assume first that we have a single puncture and mention the modifications necessary for the multi-punctured case at the end of the section. 

\subsubsection{} 
We continue using the notation of the previous section and let $\fg:=\mathrm{Lie}({G})$ and $\ft:=\mathrm{Lie}(T)$.  Let $H\in \ft^F$ be a generic  regular element and $A$ (resp. $Y$) the additive representation variety (resp. additive character variety) associated to the $G$-orbit of $H$ and the once-punctured surface group. By the additive analogue of the Frobenius Mass Formula (Theorem \ref{t:additiveFrob}), we have 
\begin{equation}\label{eq:SumY}
|\bY^F|=  \frac{|Z^F||\fg^F|^{g-1}}{|G^F|} \sum_{x\in \fg^F} |\fg_{x}^F|^{g} \mathcal{F}(1_{H}^G)(x),
\end{equation}
To proceed further, we need to recall some results on the Fourier transforms of invariant functions on $\fg^F$.

\subsection{Recollections on invariant functions}  \label{s:invariant} 
Fix a non-trivial additive character $\psi: \mathbb{F}_q\ra \bCt$ and a  non-degenerate $G$-invariant symmetric bilinear form $\kappa: \fg^F\times \fg^F\ra \mathbb{F}_q$. Since we have assumed that the ground characteristic is very good for $G$, such an invariant form exists, cf. \cite[Proposition 2.5.12]{Letellier05}.   Note that the restriction $\kappa|_{\ft}$ of $\kappa$ to $\ft^F$ is  also non-degenerate.

\subsubsection{Green function} The Green function $Q_T^G$ is the function from unipotent elements of $G^F$ to $\mathbb{Z}$ defined by
\[
Q_{T}^{G}(u):=\Big(\mathrm{Ind}_{B^F}^{G^F} 1\Big) (u)=|\mathcal{B}_u^F|.  
\]
Here, $\mathcal{B}_u\subseteq G/B$ denotes the Springer fibre associated to the unipotent element $u\in G^F$. Clearly, $Q_T^G(u)$ depends only on the $G$-orbit of $u$.

\subsubsection{Springer isomorphism} 
Since $\mathrm{char}(k)$ is very good for $G$, 
a theorem of Springer states that there exists  a $G$-equivariant isomorphism $\varpi$ from the nilpotent cone of $\fg$ to the unipotent variety of $G$, cf. \cite[\S 2.7.5]{Letellier05}. 

\subsubsection{} Let $\bC[\fg^F]^{G^F}$ denote the space of $G^F$-invariant functions on $\fg^F$. For each $x\in \fg^F$, let $1_x^G$ denote the characteristic function of the adjoint orbit $G^F\cdot x\subseteq \fg^F$. The set 
\[
{\{1_x^G\, |\, x\in \fg^F/G^F\}}
\] 
is a basis of $\bC[\fg^F]^{G^F}$. One can think of this set as the additive analogue of the set of irreducible complex characters of $G^F$.  The Fourier transform $ \mathcal{F}:\bC[\fg^F]^{G^F}\ra \bC[\fg^F]^{G^F}$ is defined by 
\[
\mathcal{F}(\phi)(x):=\sum_{y\in \fg^F} \psi(\kappa(x,y))\phi(y). 
\]

\subsubsection{}   We write $x=x_s+x_n$ for the Jordan decomposition of $x\in \fg^F$. Following \cite{Letellier05}, the Harish-Chandra induction map 
\[
{R_{\ft}^{\fg}: \bC[\ft^F]\ra \bC[\fg^F]^{G^F}}
\]
is defined by 
\[
R_{\ft}^{\fg} (f)(x) = \frac{1}{|G_{x_s}^F|} \sum_{\{g\in G^F\, | \, g\cdot x_s\in \ft^F\}} Q_{T}^{G_{x_s}}(\varpi(x_n)) f(g\cdot x_s). 
\]
Here, $G_{x_s}$ is the centraliser of $x_s$ in $G$.

\begin{defe} 
We call $x\in \fg^F$ \emph{split} if $x_s$ is $G^F$-conjugate to an element of $\ft^F$. 
\end{defe} 
It follows immediately from the definition that if $R_{\ft}^{\fg}(f)(x)$ is non-zero, then $x$ is split.

\subsection{A theorem of Kazhdan and Letellier} 
  Let $f_H: \ft^F\ra \bC$ denote the character $\psi(\kappa|_{\ft}(-,H))$. The following result was proved by Kazhdan when $\mathrm{char}(k)$ is large. It was proved by Letellier that the characteristic being very good is sufficient, cf. \cite[Theorem 7.3.3]{Letellier05}. 

\begin{thm} Suppose $H\in \ft^F$ is regular and $x\in \fg^F$. Then,  
\[
\mathcal{F}(1_{H}^G)(x) =q^{|\Phi^+|} R_{\ft}^{\fg}(f_H)(x).
\] 
\label{t:KL}
\end{thm} 

\begin{cor} If $H\in \ft^F$ is regular and $x\in \fg^F$, then $\mathcal{F}(1_{H}^G)(x) \neq 0$ only if $x$ is split.
\end{cor} 

These results can be considered as additive analogues of Theorem \ref{t:DL} and Corollary \ref{c:regular}.

\subsubsection{} 
 Let $(\fg^F)^\circ$ denote the set of split elements of $\fg^F$. The above discussions imply
\[
|\bA^F| = |\fg^F|^{g-1}  \sum_{x\in (\fg^F)^\circ}|\fg_{x}^F|^{g} \mathcal{F}(1_{H}^G)(x).
\]

\subsubsection{} Now, let $x\in (\fg^F)^\circ$ be a split element. Then  $x_s$ is $G^F$-conjugate to an element $t\in \ft^F$, well-defined up to $W$-conjugation. So, the centraliser $G_{x_s}$ is a standard Levi subgroup of $G$, well-defined up to the action of $W$. Note that since $x_n$ and $x_s$ commute, we have $x_n\in \fg_{x_s}$; thus, we can consider the orbit $G_{x_s}^F\cdot x_n\subseteq \fg_{x_s}^F$.

\subsubsection{$\fg$-types}\label{sss:fgtypes}
Consider the set of pairs $(L,\cN)$, where $L$ is a standard Levi subgroup of $G$ and $\cN$ is the $L^F$-orbit of a nilpotent element of $\fl^F$, where $\fl=\mathrm{Lie}(L)$. Note that $W$ acts on this set by conjugation. Let $\cT(\fg)$ denote the set of $W$-orbits. We call the elements of $\cT(\fg)$ \emph{$\fg$-types} and denote them by $\tau=[L,\cN]$. By the above discussion, we have a map 
\[
(\fg^F)^\circ\ra \cT(\fg).
\]
This is the additive analogue of the map \eqref{eq:typeMap}. 
\subsubsection{} 
 The following proposition implies that $\cT(\fg)$ is a finite set and independent of the ground field: 
 
\begin{prop}There exists a finite set, independent of the finite field $\mathbb{F}_q$ (of good characteristic), parameterising nilpotent $L^F$-orbits in $\fl^F$. 
\end{prop} 

\begin{proof} First of all, since $p$ is assumed to be good for $\fg$, it is also good for $\fl$. In this case,  it is known that nilpotent orbits in $\fl$ are in bijection with nilpotent orbits in $\fl(\bC)$, cf. \cite{Premet03}. Next, for any element $x\in \fl$, consider the orbit $L\cdot x$. The Frobenius fixed point $(L\cdot x)^F$ splits into $N_x$ many $L^F$-orbits, where $N_x:=|H^1(F,\pi_0(L_x))|$, cf. \cite[Proposition 4.2.14]{DM20}. Finally, observe that $N_x$ is the number of $F$-conjugacy classes in $\pi_0(L_x)$. By Corollary \ref{c:etale}, $p\nmid |\pi_0(L_x)|$. Thus, the number of $F$-conjugacy classes in $\pi_0(L_x)$ is independent of $k$. This concludes the proof of the proposition. 
\end{proof} 

One can think of the above proposition as the additive analogue of the fact that unipotent representations of $G^F$ have a characteristic free description.

\subsubsection{} Let $\tau=[L,\cN]$ be a type. Choose $n\in \cN$ and let 
\[
d(\tau):=\dim(C_L(n)). 
\] 
Clearly, this is independent of the choice of $n\in \cN$.  The following is the additive analogue of Proposition  \ref{p:mtau}:
\begin{lem} Let  $x\in (\fg^F)^\circ$ be an element of type $\tau=[L,\cN]$. Then 
\[
\dim(\fg_x)=d(\tau). 
\]
\end{lem} 

\begin{proof} 
Observe that
$\fg_{x} = \fg_{x_s}\cap \fg_{x_n} = C_{\fg_{x_s}}(x_n)$. Indeed, if an element of $\fg$ commutes with $x$, then it must commute with $x_s$ and $x_n$, because these can be written as polynomials in $x$. 
Thus, we have
\[
\dim(\fg_x)=\dim(C_{\fg_{x_s}}(x_n)) = \dim(C_{G_{x_s}}(x_n)). \qedhere
\]
\end{proof}

\subsubsection{} In view of the above discussions, we can rewrite the formula for $|\bY^F|$ as follows
 \[
|\bY^F| = \frac{|Z^F||\fg^F|^{g-1}}{|G^F|} \sum_{\tau\in \cT(\fg)} q^{d(\tau)g} \sum_{x \in \fg^F_\tau} \mathcal{F}(1_{H}^G)(x),
\] 
where $\fg^F_\tau$ denotes the set of elements in $\fg^F$ of type $\tau$. 
For ease of notation, let 
\[
\boxed{
H_\tau(q):= \displaystyle \sum_{x \in \fg^F_\tau} \mathcal{F}(1_{H}^G)(x). 
}
\]
Then we see that determining $|\bY^F|$ is reduced to computing $H_\tau(q)$. Note that $H_\tau$ is the additive analogue of $S_\tau$ considered in \S \ref{s:Stau}.

\subsection{Evaluating the sum $H_\tau$} 
First, observe that $\mathcal{F}(1_{H}^G)$ is $G^F$-invariant; thus, instead of summing over all elements in $\fg^F_\tau$, we can sum over $G^F$-orbits $\mathfrak{O}_x\subseteq \fg^F$ of elements of $x\in \fg^F_\tau$. In other words,  
\[
H_\tau(q)= \displaystyle \sum_{\mathfrak{O} \in \fg^F_\tau/G^F} |\mathfrak{O}| \mathcal{F}(1_{H}^G)(\mathfrak{O}). 
\]
Next, we show that the size of the orbit $|\mathfrak{O}|$ depends only on the type of $\mathfrak{O}$: 
\begin{lem} If $\mathfrak{O}$ has type $\tau=[L,\cN]$, then 
$\displaystyle |\mathfrak{O}| =\frac{|G^F||\cN|}{|L^F|}$. 
\end{lem} 
\begin{proof} 
Let $x=x_s+x_n\in \fg^F$.  Jordan decomposition is preserved under adjoint action; thus, for $g\in G$, $g\cdot x=g\cdot x_s+g\cdot x_n$ is the Jordan decomposition of $g\cdot x$. It follows that  $g\cdot x=x$ if and only if $g\cdot x_s=x_s$ and $g\cdot x_n=x_n$. In other words,  
\[
G_x=G_{x_s}\cap G_{x_n} = C_{G_{x_s}}(x_n).
\]
 Now if $x\in \fg^F$ has type $\tau=[L,\cN]$, then 
$C_{G_{x_s}}(x_n)=|L^F|/|\cN|$. Thus, if $\mathfrak{O}$ is the $G^F$-orbit of $x$, we conclude 
\[
|\mathfrak{O}| = \frac{|G^F|}{|G_{x}^F|}=\frac{|G^F|}{|C_{G_{x_s}}(x_n)^F|}=\frac{|G^F||\cN|}{|L^F|}. \qedhere
\]
\end{proof} 
The above lemma implies 
\[
H_\tau(q)= \frac{|G^F||\mathcal{N}|}{|L^F|} \displaystyle \sum_{\mathfrak{O} \in \fg^F_\tau/G^F} \mathcal{F}(1_{H}^G)(\mathfrak{O}). 
\]
To proceed further, we need to better understand the values of the Fourier transform. 

\subsubsection{} 
Recall that Theorem \ref{t:KL} states $\mathcal{F}(1_{H}^G)(x)=q^{|\Phi^+|} R_{\ft}^{\fg}(f_H)(x)$ and that $R_{\ft}^\fg$ is defined via a sum over the set 
\[
\mathcal{A}(x):=\{g\in G^F\, |\, g\cdot x_s\in \ft^F\}.
\]
 If we assume $x_s\in \ft^F$, then 
\[
\mathcal{A}(x) = \bigcup_{w\in W} \dot{w}G_{x_s}^F=\bigsqcup_{w\in W/W_{x_s}} \dot{w}G_{x_s}^F. 
\]
Therefore,  under this assumption, we have:
\[
R_{\ft}^{\fg}(f_H)(x) = 
\frac{Q_{T}^{G_{x_s}}(\varpi(x_n))}{|G_{x_s}^F|} \sum_{g\in \mathcal{A}(x)}  f_H(g\cdot x_s)=
\frac{Q_{T}^{G_{x_s}}(\varpi(x_n))}{ |W_{x_s}|} \sum_{w\in W} f_H(w\cdot x_s). 
\]

\subsubsection{} The above discussion implies that for a type $\tau=[L,\cN]$, we have  
\[
H_\tau=\frac{|G^F||\cN|}{|L^F|} \sum_{x \in \fg^F_\tau/G^F}  {q^{|\Phi^+|} } R_{\ft}^{\fg}(f_H)(x) = 
\frac{|G^F||\cN|}{|L^F|} \frac{ q^{|\Phi^+|}Q_{T}^{L}(\varpi(\cN))}{ |W(L)|} \sum_{w\in W} \, \, \sum_{x\in \fg^F_\tau/G^F} f_H(w\cdot x_s). 
\]
Here, $x$ is an orbit representative satisfying $x_s\in \ft^F$. 

\subsubsection{} The element $x_s$ is well-defined up to $W/W_{x_s}$-conjugacy, thus, we can now rewrite the above double sum as: 
\[
\sum_{w\in W} \sum_{x\in \fg^F_\tau/G^F}  f_H(w\cdot x_s) =\frac{|W(L)|}{|W|}
\sum_{w\in W} \sum_{{\substack{y\in \ft^F \\ [W_{y}]= [W(L)]}}}   f_H(w\cdot y) = 
 \frac{|W(L)|}{|N_W(W(L))|} \sum_{w\in W} \sum_{{\substack{y\in \ft^F \\ W_y=W(L)}}} f_H(w\cdot y).
\]

\subsubsection{} For ease of notation, let 
\[
\beta_{L,H}(q):= \sum_{\substack{y\in \ft^F\\ W_y=W(L)}} f_H(y). 
\]
In view of the fact that $|[L]|=|W/N_W(W(L))|$ (see \S \ref{s:Carter}), we conclude:  
\[
\boxed{H_\tau(q)= \frac{q^{|\Phi^+|}\, Q_{T}^{L}(\varpi(\cN))\, |\cN|\, |G^F|\, |[L]|}{|L^F|\, |W| } \sum_{w\in W} \beta_{L,w\cdot H}(q).}
\]
Thus, the problem of computing $H_\tau$ reduces to evaluating $\beta_{L,H}$. (Note that $\beta_{L,H}$ is the additive analogue of $\alpha_{L, S}$ introduced in \S\ref{s:onepunctureStau}.)

\subsection{Computing $\beta_{L,H}$} Let 
\[
\beta_{L,H}^\supseteq:=  \sum_{\substack{t\in \ft^F \\ W_t\supseteq W(L)}} f_H(t) = \sum_{t\in (\ft^F)^{W(L)}} f_H(t).  
\]
Let $\mu$ denote the M\"{o}bius function on the poset of standard Levi subgroups of $G$ ordered by inclusion.\footnote{Note that this poset is not the same as the one used in the previous section, but we have abused notation and denoted their M\"obius functions by the same symbol.} M\"{o}bius inversion gives
\[
\beta_{L,H}= \sum_{L'\supseteq L} \mu(L,L') \beta_{L',H}^\supseteq
\]

\subsubsection{} We are reduced to computing the M\"{o}bius function and the functions $\beta_{L,H}^\supseteq$. The latter can be computed as follows.  Consider the canonical quotient map from $\ft^F$ to the group of $W(L)$-coinvariants  
\[
f_\fl: \ft^F\ra (\ft^F)^{W(L)} = \ft^F/\ft^F\cap [\fl^F, \fl^F].
\]
Then Pontryagin duality implies 
\[
\beta_{L,H}^\supseteq = 
\begin{cases} |\ft^F/\ft^F\cap [\fl^F,\fl^F]| & \textrm{if $H\in [\fl^F,\fl^F]$;}\\
0 & \textrm{otherwise}. 
\end{cases}
\]

\subsubsection{} Since $H$ is generic, we have that $H\in [\fl,\fl]$ if and only if $\fl=\fg$. Note that 
\[
|\ft^F/\ft^F\cap [\fg^F,\fg^F] |=  |Z(\fg)^F|. 
\]
 Thus, $\beta_{L,H}^\supseteq$ is non-zero only if $L=G$, in which case $\ft^F/(\ft^F\cap [\fl^F, \fl^F])=Z(\fg)^F=\mathrm{Lie}(Z)^F$. We conclude 
\[
\boxed{\beta_{L,H}=q^{\dim(Z)} \mu(L,G).}
\]
 This is the additive analogue of Proposition \ref{p:alpha}. Observe that $\beta_{L,H}$ is independent of $H$ (assuming $H$ is generic).

\subsection{Conclusion of the point count for once-punctured case}
We have seen that 
\[
|\bY^F| = \frac{|Z^F||\fg^F|^{g-1}}{|G^F|} \sum_{\tau \in \cT(\fg)} q^{gd(\tau)} H_\tau(q),
\]
where 
\[
H_\tau(q) = q^{|\Phi^+|+\dim(Z)}Q_{T}^{L}(\varpi(\cN)) |\cN| \mu(L,G)  |[L]| \frac{|G^F|}{|L^F|}.
\]

\subsubsection{} It is shown in \cite[Theorem 3.10]{GoodwinRohrle} that  (under the assumption that $Z$ is connected and $p$ is good for $G$) the Green function $Q_T^G(u)$ is a polynomial in $q$, unless $G$ has a factor of type $E_8$, in which case a quadratic base change may be necessary. We can rephrase this by saying that for every type $\tau=[L, \cN]$, there exists a polynomial $f_\tau \in \mathbb{Z}[t]$ such that the size of the Springer fibre $Q_{T}^{L}(\varpi(\cN))$ over the finite field $\mathbb{F}_{q^{2r}}$ equals $f_\tau(q^{2r})$ for all positive integers $r$.

\subsubsection{} Let $s_\tau \in \mathbb{Z}[t]$ denote the counting polynomial for $C_{L}(n)$ where $n\in \cN$. Let 
\[
\gamma_G:=  \dim(Z)+(g-1)\dim(\fg)+|\Phi^+|. 
\]
Define the rational function $f_Y$ by 
\[
f_Y(t):=f_Z(t)t^{\gamma_G} \sum_{\tau \in \cT(\fg)} t^{gd(\tau)}
 \frac{f_\tau(t)}{s_\tau(t)} \mu(L,G)  |[L]| \in \mathbb{Q}(t), 
\]
where $f_Z(t)= (t-1)^{\dim(Z)}$ is the counting polynomial for $Z$ as noted in \S\ref{sss:countingpoly}.
The above discussions imply 
\[
|Y(\mathbb{F}_{q^{2r}})| = f_Y(q^{2r}),\qquad \forall r\in \mathbb{Z}_{>0}. 
\]

\subsubsection{} By \cite[Lemma 2.11]{GoodwinRohrle}, $f_Y\in \mathbb{Q}[t]$. 
Thus, $Y\times_{\Fq} \mathbb{F}_{q^2}$ is polynomial count with counting polynomial $f_Y$.   This concludes the proof of Theorem \ref{t:countY} in the single punctured case.

\subsection{Multi-punctured case}\label{s:countYPrecise} 
Let $\bY$ be the additive character variety associated to an $n$-punctured surface group   and  $(H_1,\ldots, H_n)\in \ft^n$ where each $H_i$ is regular and the tuple is generic.  Then a similar analysis as in the once-punctured case gives 
\[
|\bY^F| =  \frac{|Z^F||\fg^F|^{g-1}}{|G^F|} \sum_{\tau \in \cT(\fg)} q^{gd(\tau)} H_\tau(q),
\]
where 
\[
H_\tau(q)=  q^{n|\Phi^+|+\dim(Z)}\frac{|G^F|}{|L^F|}|\cN|\Big(Q_{T}^{L}(\varpi(\cN))\Big)^n  \bigg(\frac{|W|}{|W(L)|}\bigg)^{n-1}   |[L]|\,  \mu(L,G).
\]

\subsubsection{} 
Now, let  
\[
\gamma_G :=  \dim(Z)+(g-1)\dim(\fg)+n|\Phi^+|. 
\]
Define the rational function $f_Y$ by 
\[
f_Y(t):=f_Z(t)t^{\gamma_G} \sum_{\tau \in \cT(\fg)}t^{gd(\tau)} \bigg(\frac{|W|}{|W(L)|}\bigg)^{n-1} 
 \frac{(f_\tau(t))^n}{s_\tau(t)} \mu(L,G)  |[L]| \in \mathbb{Q}(t). 
\]
Then 
\[
|Y(\mathbb{F}_{q^{2r}})| = f_Y(q^{2r}),\qquad \forall r\in \mathbb{Z}_{>0}. 
\]

\subsubsection{} The same argument as the previous subsection implies $Y\times_{\Fq} \mathbb{F}_{q^2}$  is polynomial count with counting polynomial $f_Y$. 
This concludes the proof of Theorem \ref{t:countY}.

\subsection{Connectedness} 
Our aim now is to show that $\overline{Y}$ is connected. We have shown that $\overline{Y}$ has $E$-polynomial equal to $f_Y$. Thus, it suffices to show that $f_Y$ is monic. The only $\fg$-type contributing to the leading term of $f_Y$ is $\tau_1=[G, 0]$. The coefficient of this term is determined by the Green polynomial $Q_T^G(1)$ because
\[
H_{\tau_1}=q^{n|\Phi^+|+\dim(Z)}Q_T^G(\varpi(0)).
\] The term $Q_T^G(\varpi(0))$ is the Poincar\'e polynomial of $G$ and is known to be monic. Thus, the polynomial $c(q):=\sum_{\tau \in \mathcal{T}(\fg)} q^{gd(\tau)}H_\tau(q)$ is monic. Now, 
\[
|Y(\mathbb{F}_{q^{2r}})|=\frac{|Z(\mathbb{F}_{q^{2r}} )| |\fg(\mathbb{F}_{q^{2r}})| }{|G(\mathbb{F}_{q^{2r}})|}c(q )
\] is the quotient of two monic polynomials, therefore it is also  monic.\qed

\subsection{Euler characteristic} \label{s:EulerY}
It follows from the explicit description of $f_Y$ that 
\[
f_Y(1) = \frac{1}{r!} \frac{d^{r}}{dq^{r}} \bigg|_{q=1} \zeta(q),
\]
where $r:=\dim(T)-\dim(Z)$ is the semisimple rank of $G$ and $\zeta$ is the rational function
\[
\zeta(q) := \frac{|\mathfrak{g}(k)|^{g-1}}{q^{|\Phi^+|}P_W(q)} \sum_{\tau\in\mathcal{T}(\mathfrak{g})} q^{g d(\tau)} H_\tau(q),
\]
where $P_W(q)$ is the Poincar\'e polynomial of $W$.
We use this formula to give examples of $\chi(\overline{Y})$ in the next section. 

\newpage 

\section{Examples}\label{s:Examples} 
In this section, we give examples of Euler characteristics and counting polynomials of generic character varieties. We also give tables containing $G$-types, $\fg$-types, and the associated data required to compute the counting polynomials. Some expressions are obtained with the help of the computer; see \cite{Bailey} for the computer code using CHEVIE \cite{GHLMP96,Michel15}.
 
Throughout, $\Phi_i$ denotes the $i$th cyclotomic polynomial; thus, for instance,
\[
\Phi_1=q-1,\quad \Phi_2=q+1,\quad \Phi_3=q^2+q+1,\quad \Phi_4 = q^2+1,\quad \Phi_6 = q^2-q+1.
\]
We use the notation of \cite[\S13]{Carter93} for characters of Weyl groups, unipotent characters, and nilpotent orbits. We label Levi subgroups and endoscopy groups by their root systems, with $A_1'$ denoting short copies of $A_1$. Given an isolated standard endoscopy group $L$ of $G$, we denote by $\pi_0^L$ the number of components $|\pi_0(Z(\check{L}))|$. 
\vskip 20pt 
\subsection{Euler characteristics for $\overline{X}$ in genus $0$}
\mbox{}
\begin{figure}[h]
\centering
\begin{tabular}{|c|c|} 
\hline 
$G$ & $\chi(\overline{\bX})$ \\
\hline 
$\GL_2$ & $2^{n-4} (n-1)(n-2)$ \\ \hline 
$\GL_3$ & $2^{n-5} 3^{n-3} (n-1)(n-2) (9n^2 - 27n + 16)$ \\ \hline 
$\GL_4$ & $2^{3n-9} 3^{n-4} (n-1)(n-2) (108 n^4 - 648 n^3 + 1350 n^2 - 1129 n + 324)$ \\ \hline 
$\SO_5$ & $2^{3n-8} (n-1)(n-2) (11n^2 - 33n + 19)$ \\\hline 
$G_2$ & $2^{2n-7} 3^{n-3} (n-1)(n-2) (207n^2 - 621n + 350)$ \\
\hline 
\end{tabular}
%\caption{Euler characteristics of $\bX$ when $g=0$.} 
\end{figure}

Note that $\chi(\overline{X})=0$ when $n=1$ or $2$, in agreement with the fact that the (generic) character variety is empty in these cases. \\

\subsection{Euler characteristics of $\overline{X}$ when $(G,g)=(\SO_5, 1)$}\label{s:euler} 
\mbox{}
\vskip 10pt 
\begin{figure}[h]
\centering
\begin{tabular}{|c|ccccc|} 
\hline 
$[L]$ & $|[L]|$ & $W(L)$ & $|W(L)|$ & $|\Irr(W(L))|$ & $\nu(L)$ \\
\hline 
$[\SO_5]$ & $1$ & $D_8$ & $8$ & $5$ & $2$ \\
$[{A_1\times A_1}]$ & $1$ & $S_2\times S_2$ & $4$ & $4$ & $2$ \\
$[{A_1}]$ & $2$ & $S_2$ & $2$ & $2$ & $-4$ \\
$[{A_1'}]$ & $2$ & $S_2$ & $2$ & $2$ & $-2$ \\
$[T]$ & $1$ & $1$ & $1$ & $1$ & $8$ \\
\hline 
\end{tabular}
%\caption{Required data for $\chi(X)$ when $G=\SO_5$ and $g=1$. Using this table, we find $\boxed{\chi(\bX_{\SO_5}) = 72 \times 8^{n-1}}$.} 
\end{figure}

The above table displays the required data for computing $\chi(\overline{X})$ when $G=\SO_5$ and $g=1$. Using this table, we find ${\chi(\overline{\bX}) = 72 \times 8^{n-1}}$.

\newpage

\subsection{Euler characteristics of $\overline{X}$ when $(G,g)=(G_2, 1)$}\label{s:euler} 
\mbox{}\\

\begin{figure}[h]
\begin{tabular}{|c|ccccc|} 
\hline 
$[L]$ & $|[L]|$ & $W(L)$ & $|W(L)|$ & $|\Irr(W(L))|$ & $\nu(L)$ \\
\hline 
$[G_2]$ & $1$ & $D_{12}$ & $12$ & $6$ & $1$ \\
$[{A_2}]$ & $1$ & $S_3$ & $6$ & $3$ & $2$ \\
$[{A_1\times A_1'}]$ & $3$ & $S_2\times S_2$ & $4$ & $4$ & $1$ \\
$[{A_1}]$ & $3$ & $S_2$ & $2$ & $2$ & $-4$ \\
$[{A_1'}]$ & $3$ & $S_2$ & $2$ & $2$ & $-2$ \\
$[T]$ & $1$ & $1$ & $1$ & $1$ & $12$ \\
\hline 
\end{tabular}
%\caption{Required data for $\chi(X)$ when $G=G_2$ and $g=1$. Using this table, we find $\boxed{\chi(\bX_{G_2}) = 96 \times 12^{n-1}}$\\\\}
\end{figure}

The above table displays the required data for computing $\chi(\overline{X})$ when $G=G_2$ and $g=1$. Using this table, we find $\chi(\overline{\bX}) = 96 \times 12^{n-1}$.\\\\

\vskip10pt
\subsection{Euler characteristic of $\overline{\bY}$} \label{s:EulerY}
\mbox{}\\
\begin{figure}[h]
\centering
\begin{tabular}{|c|c|} 
\hline 
$G$ & $\chi(\overline{Y})$ \\
\hline 
$\GL_2$ & $1+2^{n-2}(4g+n-3)$ \\
\hline
$\GL_3$ & \makecell{$1+2\cdot 3^{n-1}(3g+n-3)$ \\ $+3^{n-2}2^{n+3}(144g^2+72g(n-3)  +80 -53n + 9n^2)$} \\
\hline
$\SO_5$ & \makecell{$1+2^{n-2} + 2^{2n-3}(24g+7n-22)$ \\ $+ 2^{3n-5}(96g^2 + 48g(n-3) + 53 -35n + 6n^2)$} \\
\hline
$G_2$ & \makecell{$\frac{1}{6}(6-2^n + 19\cdot 3^{n+1} + 5^n)$ \\ $+395\cdot 12^{n-2} + 5g \cdot 6^n +6^{n-2}(48-65\cdot 2^n)n$ \\ $+5\cdot 3^n4^{n-2}(16g^2 + 8g(n-3) + n^2)$} \\
\hline 
\end{tabular}
%\caption{Euler characteristics of $\bY$.} 
\end{figure}

\newpage 

\subsection{Types for $G=\mathrm{SO}_5$ and expressions for the counting polynomial $f_X$} \ 

\begin{figure}[h]
\centering
\begin{tabular}{|c|cc|ccccccc|} 
\hline
$\tau=[L,\rho]$ & $m_\tau(q)$ & $S_\tau(q)$ & $|L^F|$ & $\tilde{\rho}(1)$ & $\rho(1)$ & $|W(L)|$ & $|[L]|$ & $\pi_0^L$ & $\nu(L)$ \\
\hline
$[\SO_5,\binom{2}{}]$ 		& $q^4\Phi_1^2\Phi_2^2\Phi_4$ & $2$ & $q^4\Phi_1^2\Phi_2^2\Phi_4$ & $1$ & $1$ & $8$ & $1$ & $2$ & $2$ \\
$[\SO_5,\binom{0\ 1}{2}]$ 		& $2q^3\Phi_1^2\Phi_2^2$ & $2$ & $q^4\Phi_1^2\Phi_2^2\Phi_4$ & $\frac{1}{2}q\Phi_4$ & $1$ & $8$ & $1$ & $2$ & $2$ \\
$[\SO_5,\binom{1\ 2}{0}]$ 		& $2q^3\Phi_1^2\Phi_2^2$ & $2$ & $q^4\Phi_1^2\Phi_2^2\Phi_4$ & $\frac{1}{2}q\Phi_4$ & $1$ & $8$ & $1$ & $2$ & $2$ \\
$[\SO_5,\binom{0\ 2}{1}]$ 		& $2q^3\Phi_1^2\Phi_4$ & $2^{n+1}$ & $q^4\Phi_1^2\Phi_2^2\Phi_4$ & $\frac{1}{2}q\Phi_2^2$ & $2$ & $8$ & $1$ & $2$ & $2$ \\
$[\SO_5,\binom{0\ 1\ 2}{1\ 2}]$ 	& $\Phi_1^2\Phi_2^2\Phi_4$ & $2$ & $q^4\Phi_1^2\Phi_2^2\Phi_4$ & $q^4$ & $1$ & $8$ & $1$ & $2$ & $2$ \\
$[{A_1\times A_1},2^1\otimes2^1]$ 	& $ q^4\Phi_1^2\Phi_2^2$ & $2^n$ & $q^2\Phi_1^2\Phi_2^2$ & $1$ & $1$ & $4$ & $1$ & $4$ & $2$ \\
$[{A_1\times A_1},2^1\otimes1^2]$ 	& $q^3\Phi_1^2\Phi_2^2$ & $2^n$ & $q^2\Phi_1^2\Phi_2^2$ & $q$ & $1$ & $4$ & $1$ & $4$ & $2$ \\
$[{A_1\times A_1},1^2\otimes2^1]$ 	& $q^3\Phi_1^2\Phi_2^2$ & $2^n$ & $q^2\Phi_1^2\Phi_2^2$ & $q$ & $1$ & $4$ & $1$ & $4$ & $2$ \\
$[{A_1\times A_1},1^2\otimes1^2]$ 	& $q^2\Phi_1^2\Phi_2^2$ & $2^n$ & $q^2\Phi_1^2\Phi_2^2$ & $q^2$ & $1$ & $4$ & $1$ & $4$ & $2$ \\
$[{A_1},2^1]$ & $q^4\Phi_1^2\Phi_2$ & $-2\cdot 4^n$ & $q\Phi_1^2\Phi_2$ 	& $1$ & $1$ & $2$ & $2$ & $2$ & $-4$ \\
$[{A_1},1^2]$ & $q^3\Phi_1^2\Phi_2$ & $-2\cdot 4^n$ & $q\Phi_1^2\Phi_2$ 	& $q$ & $1$ & $2$ & $2$ & $2$ & $-4$ \\
$[{A_1'},2^1]$ & $q^4\Phi_1^2\Phi_2$ & $-4^{n}$ & $q\Phi_1^2\Phi_2$ & $1$ & $1$ & $2$ & $2$ & $1$ & $-2$ \\
$[{A_1'},1^2]$ & $q^3\Phi_1^2\Phi_2$ & $-4^{n}$ & $q\Phi_1^2\Phi_2$ & $q$ & $1$ & $2$ & $2$ & $1$ & $-2$ \\
$[T,1^1]$ & $q^4\Phi_1^2$ & $8^n$ & $\Phi_1^2$ & $1$ & $1$ & $1$ & $1$ & $1$ & $8$ \\
\hline
\end{tabular}
%\caption{The fourteen $\mathrm{SO}_5$-types.} 
\end{figure}
%\null\vfill
%\end{landscape}

\vskip20pt
The above table displays $\mathrm{SO}_5$-types and the data require for counting points on character varieties. The point counts for low values of $g$ and $n$ are given explicitly below. Since the polynomials are palindromic, we only give ``half" of them. 
\vskip 20pt
\begin{figure}[h]
\centering
\begin{tabular}{|c|c|} 
\hline
$(g,n)$ & $f_X(q)$ \\
\hline
$(0,3)$ & $2q^4+12q^3+48q^2+12q+2$ \\
\hline
$(0,4)$ & \makecell{$2q^{12}+16q^{11}+68q^{10}+208q^9+530q^8$ \\
$+1216q^7+1968q^6+\cdots+16q+2$} \\
\hline
$(1,1)$ & \makecell{$2q^8+4q^7+6q^6+16q^5+16q^4$ \\ $+16q^3+6q^2+4q+2$} \\
\hline
$(1,2)$ & \makecell{$2q^{16}+8q^{15}+16q^{14}+24q^{13}+32q^{12}+40q^{11}$ \\
$+48q^{10}+24q^9+188q^8+\cdots+8q+2$} \\
\hline
$(2,1)$ & \makecell{$2q^{28}+4q^{27}-4q^{25}-8q^{24}-12q^{23}+2q^{22}$ \\ 
$+16q^{21}+6q^{20}+52q^{19}-40q^{18}+228q^{17}$ \\ 
$-840q^{16}+708q^{15}-228q^{14}+\cdots+4q+2$} \\
\hline
\end{tabular}
%\caption{Examples of the polynomial $|\bX(\Fq)|$ when $G=\SO_5$.} 
\end{figure}

\newpage 

\begin{landscape}
\subsection{Types for $\mathfrak{so}_5$ and expressions for the counting polynomial $f_Y$} 
\mbox{}\\

%\null\vfill
\begin{figure}[h]
%\centering
\begin{tabular}{|c|cc|cccccc|} 
\hline
$\tau=[L,\cN]$ & ${d(\tau)}$ & $H_\tau(q)$ & $|L^F|$ & $|\cN|$ & $Q_T^L(\varpi(\cN))$ & $|W(L)|$ & $|[L]|$ & $\mu(L,G)$ \\
\hline
$[\SO_5,1^5]$ & ${10}$ & $q^{4n}\Phi_2^{2n}\Phi_4^n$ & $q^4\Phi_1^2\Phi_2^2\Phi_4$ & $1$ & $\Phi_2^2\Phi_4$ & $8$ & $1$ & $1$ \\
$[\SO_5,2^21^1]$ & ${6}$ & $q^{4n}\Phi_1\Phi_2^{2n+1}\Phi_4$ & $q^4\Phi_1^2\Phi_2^2\Phi_4$ & $\Phi_1\Phi_2\Phi_4$ & $\Phi_2^2$ & $8$ & $1$ & $1$ \\
$[\SO_5,3^11^2]$ & ${4}$ & $\frac{1}{2}q^{4n+1}\Phi_1\Phi_2^2\Phi_4(3q+1)^n$ & $q^4\Phi_1^2\Phi_2^2\Phi_4$ & $\frac{1}{2}q\Phi_1\Phi_2^2\Phi_4$ & $3q+1$ & $8$ & $1$ & $1$ \\
$[\SO_5,3^11^2_\star]$ & ${4}$ & $\frac{1}{2}q^{4n+1}\Phi_1^2\Phi_2^{n+1}\Phi_4$ & $q^4\Phi_1^2\Phi_2^2\Phi_4$ & $\frac{1}{2}q\Phi_1^2\Phi_2\Phi_4$ & $\Phi_2$ & $8$ & $1$ & $1$ \\
$[\SO_5,5^1]$ & ${2}$ & $q^{4n+2}\Phi_1^2\Phi_2^2\Phi_4$ & $q^4\Phi_1^2\Phi_2^2\Phi_4$ & $q^2\Phi_1^2\Phi_2^2\Phi_4$ & $1$ & $8$ & $1$ & $1$ \\
$[A_1,1^2]$ & ${4}$ & $-2\cdot 4^{n-1} q^{4n+3} \Phi_2^{n+1}\Phi_4$ & $q\Phi_1^2\Phi_2$ & $1$ & $\Phi_2$ & $2$ & $2$ & $-1$ \\
$[A_1,2^1]$ & ${2}$ & $-2\cdot 4^{n-1} q^{4n+3} \Phi_1\Phi_2^2\Phi_4$ & $q\Phi_1^2\Phi_2$ & $\Phi_1\Phi_2$ & $1$ & $2$ & $2$ & $-1$ \\
$[A_1',1^2]$ & ${4}$ & $-2\cdot 4^{n-1} q^{4n+3} \Phi_2^{n+1}\Phi_4$ & $q\Phi_1^2\Phi_2$ & $1$ & $\Phi_2$ & $2$ & $2$ & $-1$ \\
$[A_1',2^1]$ & ${2}$ & $-2\cdot 4^{n-1}q^{4n+3}\Phi_1\Phi_2^2\Phi_4$ & $q\Phi_1^2\Phi_2$ & $\Phi_1\Phi_2$ & $1$ & $2$ & $2$ & $-1$ \\
$[T,0]$ & ${2}$ & $3 \cdot 8^{n-1} q^{4n+4}\Phi_2^2\Phi_4$ & $\Phi_1^2$ & $1$ & $1$ & $1$ & $1$ & $3$ \\
\hline
\end{tabular}
%\caption{The ten $\mathfrak{so}_5$-types.} 
\end{figure}
%\null\vfill

The above table displays $\mathfrak{so}_5$-types and the data require for counting points on the additive character varieties. The point counts for low values of $g$ and $n$ are given explicitly below.

\begin{figure}[h]
\centering
\begin{tabular}{|c|c|} 
\hline
$(g,n)$ & $f_Y(q)$ \\
\hline
$(0,3)$ & $q^4+6q^3+20q^2$ \\
\hline
$(0,4)$ & \makecell{$q^{12}+8q^{11}+33q^{10}+96q^9+223q^8+440q^7+548q^6$} \\
\hline
$(1,1)$ & $q^8+2q^7+4q^6+4q^5+q^4$ \\
\hline
$(1,2)$ & \makecell{$q^{16}+4q^{15}+9q^{14}+16q^{13}+25q^{12}$ \\ $+36q^{11}+36q^{10}+16q^9+q^8$} \\
\hline
$(2,1)$ & \makecell{$q^{28}+2q^{27}+3q^{26}+4q^{25}+5q^{24}+6q^{23}+8q^{22}$ \\ 
$+10q^{21}+11q^{20}+10q^{19}+7q^{18}+4q^{17}+q^{16}$
} \\
\hline
\end{tabular}
%\caption{Examples of the polynomial $|\bY(\Fq)|$ when $\fg=\mathfrak{so}_5$.} 
\end{figure}
%\null\vfill
%\end{landscape}

\end{landscape}

\begin{landscape}
\subsection{Types for $G=G_2$}\ 
\null\vfill
\begin{figure}[h]
\centering
\begin{tabular}{|c|cc|ccccccc|} 
\hline
$\tau=[L,\rho]$ & $m_\tau(q)$ & $S_\tau(q)$ & $|L^F|$ & $\rho(1)$ & $\tilde{\rho}(1)$ & $|W(L)|$ & $|[L]|$ & $\pi_0^L$ & $\nu(L)$ \\
\hline 
$[G_2,\phi_{1,0}]$ & $q^6\Phi_1^2\Phi_2^2\Phi_3\Phi_6$ & $1$ & $q^6\Phi_1^2\Phi_2^2\Phi_3\Phi_6$ & $1$ & $1$ & $12$ & $1$ & $1$ & $1$ \\
$[G_2,\phi_{1,3}']$ & $3q^5\Phi_1^2\Phi_2^2$ & $1$ & $q^6\Phi_1^2\Phi_2^2\Phi_3\Phi_6$ & $\frac{1}{3}q\Phi_3\Phi_6$ & $1$ & $12$ & $1$ & $1$ & $1$ \\
$[G_2,\phi_{1,3}'']$ & $3q^5\Phi_1^2\Phi_2^2$ & $1$ & $q^6\Phi_1^2\Phi_2^2\Phi_3\Phi_6$ & $\frac{1}{3}q\Phi_3\Phi_6$ & $1$ & $12$ & $1$ & $1$ & $1$ \\
$[G_2,\phi_{2,1}]$ & $6q^5\Phi_1^2\Phi_6$ & $2^n$ & $q^6\Phi_1^2\Phi_2^2\Phi_3\Phi_6$ & $\frac{1}{6}q\Phi_2^2\Phi_3$ & $2$ & $12$ & $1$ & $1$ & $1$ \\
$[G_2,\phi_{2,2}]$ & $2q^5\Phi_1^2\Phi_3$ & $2^n$ & $q^6\Phi_1^2\Phi_2^2\Phi_3\Phi_6$ & $\frac{1}{2}q\Phi_2^2\Phi_6$ & $2$ & $12$ & $1$ & $1$ & $1$ \\
$[G_2,\phi_{1,6}]$ & $\Phi_1^2\Phi_2^2\Phi_3\Phi_6$ & $1$ & $q^6\Phi_1^2\Phi_2^2\Phi_3\Phi_6$ & $q^6$ & $1$ & $12$ & $1$ & $1$ & $1$ \\
$[{A_2},3^1]$ & $q^6\Phi_1^2\Phi_2\Phi_3$ & $2^n$ & $q^3\Phi_1^2\Phi_2\Phi_3$ & $1$ & $1$ & $6$ & $1$ & $3$ & $2$ \\
$[{A_2},2^11^1]$ & $q^5\Phi_1^2\Phi_3$ & $4^n$ & $q^3\Phi_1^2\Phi_2\Phi_3$ & $q\Phi_2$ & $2$ & $6$ & $1$ & $3$ & $2$ \\
$[{A_2},1^3]$ & $q^3\Phi_1^2\Phi_2\Phi_3$ & $2^n$ & $q^3\Phi_1^2\Phi_2\Phi_3$ & $q^3$ & $1$ & $6$ & $1$ & $3$ & $2$ \\
$[{A_1\times A_1'},2^1\otimes2^1]$ & $q^6\Phi_1^2\Phi_2^2$ & $3^n$ & $q^2\Phi_1^2\Phi_2^2$ & $1$ & $1$ & $4$ & $3$ & $2$ & $1$ \\
$[{A_1\times A_1'},2^1\otimes1^2]$ & $q^5\Phi_1^2\Phi_2^2$ & $3^n$ & $q^2\Phi_1^2\Phi_2^2$ & $q$ & $1$ & $4$ & $3$ & $2$ & $1$ \\
$[{A_1\times A_1'},1^2\otimes2^1]$ & $q^5\Phi_1^2\Phi_2^2$ & $3^n$ & $q^2\Phi_1^2\Phi_2^2$ & $q$ & $1$ & $4$ & $3$ & $2$ & $1$ \\
$[{A_1\times A_1'},1^2\otimes1^2]$ & $q^4\Phi_1^2\Phi_2^2$ & $3^n$ & $q^2\Phi_1^2\Phi_2^2$ & $q^2$ & $1$ & $4$ & $3$ & $2$ & $1$ \\
$[{A_1},2^1]$ & $q^6\Phi_1^2\Phi_2$ & $-2\cdot 6^n$ & $q\Phi_1^2\Phi_2$ & $1$ & $1$ & $2$ & $3$ & $1$ & $-4$ \\
$[{A_1},1^2]$ & $q^5\Phi_1^2\Phi_2$ & $-2\cdot 6^n$ & $q\Phi_1^2\Phi_2$ & $q$ & $1$ & $2$ & $3$ & $1$ & $-4$ \\
$[{A_1'},2^1]$ & $q^6\Phi_1^2\Phi_2$ & $-6^n$ & $q\Phi_1^2\Phi_2$ & $1$ & $1$ & $2$ & $3$ & $1$ & $-2$ \\
$[{A_1'},1^2]$ & $q^5\Phi_1^2\Phi_2$ & $-6^n$ & $q\Phi_1^2\Phi_2$ & $q$ & $1$ & $2$ & $3$ & $1$ & $-2$ \\
$[T,1^1]$ & $q^6\Phi_1^2$ & $12^n$ & $\Phi_1^2$ & $1$ & $1$ & $1$ & $1$ & $1$ & $12$ \\
\hline
\end{tabular}
%\caption{The eighteen ${G}_2$-types.} 
\end{figure}
\null\vfill
\end{landscape}

\begin{landscape}
\subsection{Types for $\fg=\fg_2$}
\mbox{}\\

\begin{figure}[h]
\centering
\begin{tabular}{|c|cc|cccccc|} 
\hline
$\tau=[L,\cN]$ & $q^{d(\tau)}$ & $H_\tau(q)$ & $|L^F|$ & $|\cN|$ & $Q_T^L(\varpi(\cN))$ & $|W(L)|$ & $|[L]|$ & $\mu(L,G)$ \\
\hline
$[G_2,0]$ & $q^{14}$ & $q^{6n} \Phi_2^{2n}\Phi_3^n\Phi_6^n$ & $q^6\Phi_1^2\Phi_2^2\Phi_3\Phi_6$ & $1$ & $\Phi_2^2\Phi_3\Phi_6$ & $12$ & $1$ & $1$ \\
$[G_2,A_1]$ & $q^8$ & $q^{6n} \Phi_1\Phi_2^{n+1}\Phi_3^{n+1}\Phi_6$ & $q^6\Phi_1^2\Phi_2^2\Phi_3\Phi_6$ & $\Phi_1\Phi_2\Phi_3\Phi_6$ & $\Phi_2\Phi_3$ & $12$ & $1$ & $1$ \\
$[G_2,\tilde{A_1}]$ & $q^6$ & $q^{6n+2}\Phi_1\Phi_2^{n+1}\Phi_3\Phi_6 (2q+1)^n$ & $q^6\Phi_1^2\Phi_2^2\Phi_3\Phi_6$ & $q^2\Phi_1\Phi_2\Phi_3\Phi_6$ & $(2q+1)\Phi_2$ & $12$ & $1$ & $1$ \\
$[G_2,G_2(a_1)_{1^3}]$ & $q^4$ & $\frac{1}{6}q^{6n+2} \Phi_1^2\Phi_2^2\Phi_3\Phi_6 (4q+1)^n$ & $q^6\Phi_1^2\Phi_2^2\Phi_3\Phi_6$ & $\frac{1}{6}q^2\Phi_1^2\Phi_2^2\Phi_3\Phi_6$ & $4q+1$ & $12$ & $1$ & $1$ \\
$[G_2,G_2(a_1)_{2^11^1}]$ & $q^4$ & $\frac{1}{2}q^{6n+2}\Phi_1^2\Phi_2^2\Phi_3\Phi_6 (2q+1)^n$ & $q^6\Phi_1^2\Phi_2^2\Phi_3\Phi_6$ & $\frac{1}{2}q^2\Phi_1^2\Phi_2^2\Phi_3\Phi_6$ & $2q+1$ & $12$ & $1$ & $1$ \\
$[G_2,G_2(a_1)_{3^1}]$ & $q^4$ & $\frac{1}{3}q^{6n+2}\Phi_1^2\Phi_2^{n+2}\Phi_3\Phi_6$ & $q^6\Phi_1^2\Phi_2^2\Phi_3\Phi_6$ & $\frac{1}{3}q^2\Phi_1^2\Phi_2^2\Phi_3\Phi_6$ & $\Phi_2$ & $12$ & $1$ & $1$ \\
$[G_2,G_2]$ & $q^2$ & $q^{6n+4}\Phi_1^2\Phi_2^2\Phi_3\Phi_6$ & $q^6\Phi_1^2\Phi_2^2\Phi_3\Phi_6$ & $q^4\Phi_1^2\Phi_2^2\Phi_3\Phi_6$ & $1$ & $12$ & $1$ & $1$ \\
$[A_1,1^2]$ & $q^4$ & $-3\cdot 6^{n-1}q^{6n+5}\Phi_2^{n+1}\Phi_3\Phi_6$ & $q\Phi_1^2\Phi_2$ & $1$ & $\Phi_2$ & $2$ & $3$ & $-1$ \\
$[A_1,2^1]$ & $q^2$ & $-3\cdot 6^{n-1}q^{6n+5}\Phi_1\Phi_2^2\Phi_3\Phi_6$ & $q\Phi_1^2\Phi_2$ & $\Phi_1\Phi_2$ & $1$ & $2$ & $3$ & $-1$ \\
$[A_1',1^2]$ & $q^4$ & $-3\cdot 6^{n-1}q^{6n+5}\Phi_2^{n+1}\Phi_3\Phi_6$ & $q\Phi_1^2\Phi_2$ & $1$ & $\Phi_2$ & $2$ & $3$ & $-1$ \\
$[A_1',2^1]$ & $q^2$ & $-3\cdot 6^{n-1}q^{6n+5}\Phi_1\Phi_2^2\Phi_3\Phi_6$ & $q\Phi_1^2\Phi_2$ & $\Phi_1\Phi_2$ & $1$ & $2$ & $3$ & $-1$ \\
$[T,0]$ & $q^2$ & $5\cdot 12^{n-1} q^{6n+6}\Phi_2^2\Phi_3\Phi_6$ & $\Phi_1^2$ & $1$ & $1$ & $1$ & $1$ & $5$ \\
\hline
\end{tabular}
%\caption{The twelve $\mathfrak{g}_2$-types.} 
\end{figure}
\null\vfill
\end{landscape}

\newpage 
%\begin{landscape}
%\null\vfill

\subsection{Counting polynomials for $X$ and $Y$ when $G=G_2$} 
\mbox{}\\

\begin{figure}[h]
\centering
\begin{tabular}{|c|c|} 
\hline
$(g,n)$ & $f_X(q)$ \\
\hline
$(0,3)$ & $q^{8}+6q^{7}+20q^{6}+58q^{5}+180q^{4}+58q^{3}+20q^{2}+6q+1$ \\
\hline
$(0,4)$ & \makecell{
$q^{20}+8q^{19}+34q^{18}+104q^{17}+259q^{16}+560q^{15}+1106q^{14}$ \\ 
$+2080q^{13}+3890q^{12}+7440q^{11}+11444q^{10} + \cdots + 8q + 1$} \\
\hline
$(1,1)$ & \makecell{
$q^{12}+2q^{11}+2q^{10}+4q^{9}+9q^{8}+26q^{7}+8q^{6}+\cdots+2q+1$
} \\
\hline
$(1,2)$ & \makecell{
$q^{24}+4q^{23}+8q^{22}+12q^{21}+16q^{20}+20q^{19}+26q^{18}+36q^{17}$ \\
$+57q^{16}+88q^{15}+198q^{14}-240q^{13}+700q^{12}+\cdots+4q+1$
} \\
\hline
$(2,1)$ & \makecell{
$q^{40}+2q^{39}-2q^{37}-q^{36}-3q^{34}-6q^{33}+8q^{31}+7q^{30}-4q^{28}$ 
\\ $-2q^{27}-9q^{26}+482q^{25}-2885q^{24}+10278q^{23}-23499q^{22}$ \\
$+35928q^{21}-40590q^{20}+\cdots+2q+1$
} \\
\hline
\end{tabular}
%\caption{Examples of the polynomial $|\bX(\FF_q)|$ when $G=G_2$.} 
\end{figure}
%\null\vfill

\vskip 20pt 

\begin{figure}[h] 
\centering
\begin{tabular}{|c|c|} 
\hline
$(g,n)$ & $f_Y(q)$ \\
\hline
$(0,3)$ & $q^8+6q^7+19q^6+45q^5+99q^4$ \\
\hline
$(0,4)$ & \makecell{$q^{20}+8q^{19}+33q^{18}+96q^{17}+225q^{16}+456q^{15}$ \\
$+831q^{14}+1392q^{13}+2191q^{12}+3300q^{11}+3498q^{10}$
} \\
\hline
$(1,1)$ & $q^{12}+2q^{11}+3q^{10}+5q^9+9q^8+8q^7+2q^6$ \\
\hline
$(1,2)$ & \makecell{$q^{24}+4q^{23}+9q^{22}+16q^{21}+25q^{20}+36q^{19}+49q^{18}$ \\
$+64q^{17}+85q^{16}+110q^{15}+99q^{14}+40q^{13}+2q^{12}$} \\
\hline
$(2,1)$ & \makecell{$q^{40}+2q^{39}+3q^{38}+4q^{37}+5q^{36}+6q^{35}+7q^{34}+8q^{33}+9q^{32}+11q^{31}$ \\
$+13q^{30}+15q^{29}+18q^{28}+20q^{27}+21q^{26}+18q^{25}+12q^{24}+6q^{23}+q^{22}$} \\
\hline
\end{tabular}
%\caption{Examples of the polynomial $|\bY(\Fq)|$ when $\fg=\fg_2$.} 
\label{figure:addg2}
\end{figure}
%\null\vfill
%\end{landscape}

%\begin{landscape}

\newpage 
\subsection{Counting polynomials when $(g,n)=(0,3)$}\label{ss:nonneg}
\mbox{}\\

%\null\vfill
\begin{figure}[h]
\centering
\begin{tabular}{|c|c|} 
\hline 
$G$ & $f_X(q)$ \\
\hline 
$B_3$ & \makecell{$2q^{12}+18q^{11}+88q^{10}+312q^9+952q^8+2478q^7+4980q^6+ \cdots +18q + 2$} \\
\hline
$C_3$ & \makecell{$2q^{12}+18q^{11}+88q^{10}+320q^9+970q^8+2506q^7+5060q^6 + \cdots + 18q + 2$} \\
\hline
$D_4$ & \makecell{$4q^{16}+48q^{15}+308q^{14}+1408q^{13}+5140q^{12}+16176q^{11}$ \\
$+43748q^{10}+96256q^9+152864q^8+\cdots+48q+4$
} \\
\hline
$F_4$ & \makecell{$q^{40}+12q^{39}+77q^{38}+352q^{37}+1287q^{36}+4004q^{35}+11010q^{34}+27444q^{33}$ \\
$+63155q^{32}+136096q^{31}+277849q^{30}+542780q^{29}+1023705q^{28}+1879344q^{27}$ \\
$+3384741q^{26}+6009812q^{25}+10498044q^{24}+17873468q^{23}+29224321q^{22}$ \\ 
$+44790488q^{21}+58508548q^{20}+\cdots+12q+1$} \\
\hline
$E_6$ & \makecell{$3q^{60}+54q^{59}+510q^{58}+3366q^{57}+17442q^{56}$ \\
$+75579q^{55}+284829q^{54}+958746q^{53}$ \\
$+2938005q^{52}+8314473q^{51}+21970542q^{50}$ \\
$+54685197q^{49}+129125661q^{48}+290973411q^{47}$ \\
$+628951404q^{46}+1310014803q^{45}+2640078312q^{44}$ \\
$+5167269447q^{43}+9854714367q^{42}+18363440685q^{41}$ \\
$+33501157887q^{40}+59902887102q^{39}+104998780206q^{38}+180252368460q^{37}$ \\
$+302421156681q^{36}+493955088993q^{35}+780185905473q^{34}+1178431874241q^{33}$ \\ 
$+1671208058577q^{32}+2158923553011q^{31}+2408708043594q^{30}+\cdots+510q^2+54q+3$} \\
\hline 
\end{tabular}
%\caption{Formulas for $|X(\FF_q)|$ when $(g,n)=(0,3)$.} 
\end{figure}
%\null\vfill
%\end{landscape}

\vskip 20pt

%\begin{landscape}
%\null\vfill
\begin{figure}[h]
\centering
\begin{tabular}{|c|c|} 
\hline 
$G$ & $f_Y(q)$ \\
\hline 
$B_3$ & $q^{12}+9q^{11}+43q^{10}+147q^9+412q^8+948q^7+1535q^6
$ \\
\hline
$C_3$ & $q^{12}+9q^{11}+43q^{10}+148q^9+413q^8+946q^7+1533q^6
$ \\
\hline
$D_4$ & \makecell{$q^{16}+12q^{15}+76q^{14}+340q^{13}+1206q^{12}+3668q^{11}+9451q^{10}+19080q^9+24492q^8$} \\
\hline
$F_4$ & \makecell{$q^{40}+12q^{39}+76q^{38}+340q^{37}+1210q^{36}+3652q^{35}+9722q^{34}+23428q^{33}+52049q^{32}$ \\
$+108072q^{31}+211964q^{30}+396184q^{29}+711295q^{28}+1235704q^{27}+2091762q^{26}$ \\ 
$
+3459608q^{25}+5550351q^{24}+8475196q^{23}+11932547q^{22}+14602160q^{21}+12631801q^{20}$} \\
\hline
$E_6$ & \makecell{$q^{60}+18q^{59}+169q^{58}+1104q^{57}+5644q^{56}+24070q^{55}+89110q^{54}+294202q^{53}$ \\
$+883118q^{52}+2445142q^{51}+6314277q^{50}+15342436q^{49}+35328016q^{48}+77552468q^{47}$ \\
$+163141890q^{46}+330385400q^{45}+646787466q^{44}+1228563901q^{43}+2271409786q^{42}$ \\
$+4097165169q^{41}+7220215707q^{40}+12432586601q^{39}+20893248670q^{38}$ \\ 
$+34175389336q^{37}+54155977909q^{36}+82488035626q^{35}+119164974059q^{34}$ \\
$+159544431246q^{33}+189727222032q^{32}+183584161672q^{31}+112360075923q^{30}$} \\
\hline 
\end{tabular}
%\caption{Formulas for $|Y(\FF_q)|$ when $(g,n)=(0,3)$.} 
 \label{figure:gndiverse}
\end{figure}
%\null\vfill
%\end{landscape}

\end{document}